\documentclass[11pt]{amsart}

\usepackage{amsmath,amssymb,amsthm}
\usepackage{color, tikz}
\usetikzlibrary{cd}
\usepackage{tabularx}
\usepackage{float}
\usepackage[linesnumbered,lined,boxed,commentsnumbered]{algorithm2e}
\usepackage{textcomp}
\usetikzlibrary{patterns}

\makeatletter
\newcommand{\addresseshere}{%
  \enddoc@text\let\enddoc@text\relax
}
\makeatother

\usepackage[margin=1in]{geometry}

\theoremstyle{plain}
\newtheorem{theorem}{Theorem}[section]
\newtheorem{proposition}[theorem]{Proposition}

\newtheorem{conjecture}[theorem]{Conjecture}

\newtheorem{lemma}[theorem]{Lemma}
\newtheorem{corollary}[theorem]{Corollary}

\theoremstyle{definition}
\newtheorem{remark}[theorem]{Remark}
\newtheorem{definition}[theorem]{Definition}
\newtheorem{observation}[theorem]{Observation}
\newtheorem{example}[theorem]{Example}
\newtheorem{question}[theorem]{Question}

\definecolor{green}{RGB}{34, 139, 34}
\newcommand{\jnote}[1]{\par \noindent
  \framebox{\begin{minipage}[c]{0.95 \textwidth}\color{green} JULIE'S NOTE:
      #1 \color{black}\end{minipage}}\par}
\newcommand\commentout[1]{}

\begin{document}

\title{$2$-matching complexes}

\author{Julianne Vega}
\address{Department of Mathematics\\
         Kennesaw State University\\
         Marietta, GA 30060}
\email{jvega30@kennesaw.edu}

\date{\today}


\begin{abstract}
A $2$-matching complex is a simplicial complex which captures the relationship between $2$-matchings of a graph. In this paper, we will use discrete Morse Theory and the Matching Tree Algorithm to prove homotopical results. We will consider a class of graphs for which the homotopy type of the $2$-matching complex transforms from a sphere to a point with the addition of leaves. We end the paper by defining $k$-matching sequences and looking at the $1$- and $2$-matching complexes of wheel graphs and perfect caterpillar graphs. 

\end{abstract}

\maketitle

\section{Introduction} 

Matchings and matching complexes are objects that have been well studied, for example \cite{Bjorner_etal, Jakob, Shareshian_Wachs, Ziegler}. Matching complexes were introduced in the 70's through work done by Brown and Quillen as a way to study the structure of subgroups and provide interesting connections to several areas in mathematics. A \emph{matching complex} of a graph $G$, denoted $M_1(G)$, is a simplicial complex with vertices given by edges of $G$ and faces given by matchings of $G$, where a matching is a subset of edges $H \subseteq E(G)$ such that any vertex $ v \in V(H)$ has degree at most $1$. 
Some matching complexes that have been studied in detail are the full matching complex $M_1(K_n)$, where $K_n$ is the complete graph on $n$ vertices, and the chessboard complex $M_1(K_{m,n})$, where $K_{m,n}$ is the complete bipartite graph with block size $m$ and $n$. Results about $M_1(K_n)$ and $M_1(K_{m,n})$ include connectivity bounds and rational homology. For a general survey on matching complexes see \cite{Wachs}. The homotopy type of matching complexes is a bit more mysterious. The homotopy type of matching complexes for paths and cycles \cite{Kozlov}, for forests \cite{Marietti_Testa_forests}, and for the $\left( \lfloor \frac{n+m+1}{3} -1\rfloor \right) $ skeleton of  $M_1(K_{m,n})$ for all $m,n$ and $M_1(K_{m,n})$ when $2m-1 \leq n$~\cite{Ziegler} is known to be either a point, sphere, or wedge of spheres, but beyond these classes the homotopy type of matching complexes is unclear. In fact, we know that torsion arises in higher homology groups of $M_1(K_n)$ and $M_1(K_{m,n})$ \cite{Shareshian_Wachs}.

 In \cite{Jakob}, Jonsson defines the bounded degree complex $BD_n^\lambda(G)$ with $\lambda = (\lambda_1, \lambda_2, \dots ,\lambda_n)$ to be the complex of subgraphs of a graph $G$ with $n$ vertices such that the degree of vertex $x_i$ is at most $\lambda_i$ in the subgraph, which  is a natural generalization of matching complexes. When $\lambda = (d, \dots, d)$ we write $BD_n^d(G) := BD_n^{(d,d,..,d)}(G)$. The bounded degree complex $BD_n^1(K_n)$ is the matching complex on complete graphs, that is $M_1(K_n)$.  For $d \geq 2,~BD_n^d(G)$ is the \emph{$d$-matching complex} on $G$ with $0$-simplices given by edges in $G$ and faces by $d$-matchings in $G$, where a $d$-matching is a subset of edges $H \subseteq E(G)$ such that any vertex $v \in V(H)$ has degree at most $d$. Bounded degree complexes are generalizations of matching complexes that involve relaxing the incidence conditions on the vertices. Such a generalization may provide insight into the complexity of matching complexes. For example, in Section~\ref{sec:cat} we use bounded degree complexes to inductively study $k$-matching complexes. Jonsson primarily focuses on the connectivity of $BD_n^\lambda(K_n)$ the bounded degree complex of the complete graph and $\overline{BD}_n^\lambda(K_n)$ the bounded degree complex of the complete graphs with loops allowed at each vertex.  For a further survey of bounded degree complexes see~\cite{Wachs}.

In Section~\ref{sec:contractible}, we connect our results to these connectivity results. The focus of this paper will be the topology of $M_2(G) := BD_n^2(G)$ which we call the $2$-matching complex of $G$. Since a matching of $G$ is also a $2$-matching of $G$, the matching complex of $G$ is a subcomplex of the $2$-matching complex of $G$, with $M_1(G) \subset M_2(G)$. 

\subsection{Our contributions}
In this paper, we will use discrete Morse Theory and the Matching Tree Algorithm to prove homotopical results. In Section~\ref{sec:background}, we provide the necessary combinatorial and topological background for these techniques. In Section~\ref{sec:contractible}, we take a preliminary look at $2$-matching complexes and consider a class of graphs for which the homotopy type of the $2$-matching complex is contractible. Then, we look at graphs whose homotopy type of the $2$-matching complex changes from a sphere to a point with the addition of leaves. We end this section with a constructible algorithm to maximize the number of additional leaves that can be added to a certain family of graphs without changing the homotopy type of $M_2(G)$. In Section~\ref{sec:wheel}, we define $k$-matching sequences and look at wheel graphs as a first example. We conclude with perfect caterpillar graphs and future directions.



\section{Background}\label{sec:background}

\begin{definition}
An \emph{(abstract) simplicial complex} $\Delta$ on a set $X$ is a collection of subsets of $X$ such that
\begin{enumerate}
\item[(i)]$ \emptyset \in \Delta$ 
\item[(ii)] If $\sigma \in \Delta$ and $\tau \subseteq \sigma$, then $\tau \in \Delta$.
\end{enumerate}
\end{definition} 

The elements of a simplicial complex are called \emph{faces} and an \emph{$n-simplex$} is the collection of all subsets of $[n+1]$. A subcomplex $\Delta'$ of a complex $\Delta$ is a subcollection of $\Delta$ which satisfies (i) and (ii). For disjoint simplicial complexes $\Delta$ and $\Delta'$, the topological \emph{join} is $\Delta \ast \Delta' = \{ \sigma \cup \sigma': \sigma \in \Delta, \sigma' \in \Delta' \}$. A simplicial complex $\Delta$ is said to be a \emph{cone} with \emph{cone point} $\{x\} \in \Delta$ if for every face $\sigma \in \Delta$ we have $\sigma \cup \{x\} \in \Delta$, that is the simplicial complex $\Delta' \ast x$ for some $\Delta'$.  Note that every cone is contractible. The suspension of a space $\Delta$ is denoted $\Sigma(\Delta)$ and is the join of $\Delta$ with $2$ discrete points. A simplicial complex $\Delta$ is $k$-connected if the higher homotopy groups $\pi_i$ are trivial for $0<i< k$. The connectivity of $\Delta$ is the largest integer $k$ for which $\Delta$ is $k$-connected. 

We will let $G$ be a finite, simple graph with vertex set $V(G)$ and edge set $E(G)$. For a vertex $v \in V(G)$, the \emph{degree of $v$}, deg($v$) is the number of edges incident to $v$. If $V(G) \cap V(H) = \{x\}$, the 
\emph{wedge sum} $G \underset{x}{\vee} H$ of $G$ with $H$ over $x$ is the graph with vertex set $V(G) \cup V(H)$ and edge set $E(G) \cup E(H)$. Let $\{u,v\} \in E(G)$ and $w$ a new vertex not in $V(G)$. The \emph{subdivision of $\{u,v\} \in G$} is obtained by deleting $\{u,v\}$ and adding $w$ to $V(G)$ and $\{u,w\}, \{w,v\}$ to $ E(G)$. For an edge set $H \subseteq E(G)$, let $V(H)$ denote the set of vertices supported by $H$. That is, $V(H) := \bigcup\limits_{e \in H}V(e)$. A vertex $v \in V(G)$ is a leaf if its neighborhood contains exactly $1$ vertex. For a graph $G$ with $v \in V(G)$, \emph{attaching a leaf to $v$ in $G$} refers to the process of adding a new vertex $w$ to $V(G)$ and $\{v,w\}$ to $E(G)$. Given a graph $G$ with $2$ leaves $u,v$ and edges $\{v_1,u\}$ and $\{v_2,v\}$, define $G_{(u,v)}$ to be the graph obtained by identifying $u$ and $v$, labeled $uv$. That is $E(G_{(u,v)}) = E(G) \smallsetminus \{\{v_1,u\}, \{v_2,v\}\} \cup \{\{v_1,uv\}, \{v_2,uv\}\}$ and $V(G_{(u,v)}) = V(G) \smallsetminus \{u,v\} \cup \{uv\}$. 

\begin{definition} 
For a graph $G = (V(G), E(G))$ with max degree $3$, the \emph{clawed graph of $G$}, denoted $CG$ is the graph obtained by subdividing every $e \in E(G)$ and attaching a (possibly empty) set of leaves to every $v \in V(G)$ so that deg$(v) = 3$ for all $v$. The graph $G$ is called the \emph{core} of $CG$.  See Example~\ref{fig:claw_graph} for an example.
\end{definition}
If $|E(G)|$ and $|V_{\leq 2}(G)|$ denote the number of edges and the number of vertices with degree less than or equal to $2$ in a graph $G$, respectively, and $L$ is the number of leaves of $G$, the process of clawing  $G$ introduces $|E(G)| + |V_{\leq 2}(G)| + L$ new vertices and $|V_{\leq 2}(G)| + L$ new edges.
\begin{example}[Clawing a graph]
See figure~\ref{fig:claw_graph}. (A) Begin with a graph $G$, (B) Subdivide each edge (depicted with open circles), (C) attach a set of leaves to each vertex of $G$ so that deg($v$) $= 3$ for all $v \in V(G)$. Notice $|E(G)| = 4 = |V_{\leq2}(G)|$ and $L = 3$ so the total number of vertices added is $11$ and the total number of new (leaf) edges is $7$. 
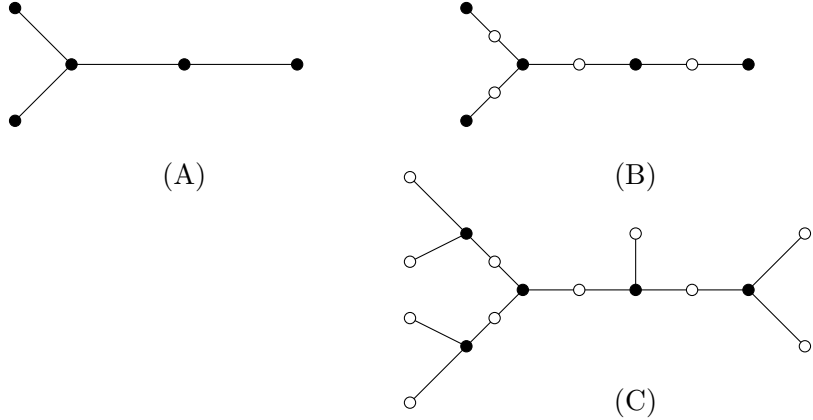
\begin{figure}[h]
\begin{center}
\begin{tikzpicture}[scale = .75]

\filldraw (-1, 5) circle (0.1cm);
\filldraw (-1, 3) circle (0.1cm);
\filldraw (2, 4) circle (0.1cm);
\filldraw (0, 4) circle (0.1cm);
\filldraw (4,4) circle (0.1cm);

\draw (-1, 5) -- (0, 4) -- (2, 4) --(4,4);
\draw (-1, 3) -- (0, 4) ;

\node at (2, 2) {(A)};


\draw  (7,5) -- (8, 4);
\draw  (7,3) -- (8, 4) -- (12,4);

\filldraw (7, 5) circle (0.1cm);
\filldraw (7, 3) circle (0.1cm);
\filldraw[fill=white] (7.5, 3.5) circle (0.1cm);
\filldraw[fill=white] (7.5, 4.5) circle (0.1cm);
\filldraw (8, 4) circle (0.1cm);
\filldraw[fill=white] (9, 4) circle (0.1cm);
\filldraw (10,4) circle (0.1cm);
\filldraw[fill=white] (11,4) circle (0.1cm);
\filldraw (12,4) circle (0.1cm);

\node at (10, 2) {(B)};

\draw (6, 2) -- (7,1) -- (8, 0);
\draw (6, -2) -- (7,-1) -- (8, 0) -- (12,0) -- (13,1);
\draw (6, 0.5) -- (7,1) ;
\draw (6, -0.5) -- (7,-1) ;

\draw (10,0) -- (10,1) ;
\draw (12,0) -- (13,-1) ;

\filldraw (7, 1) circle (0.1cm);
\filldraw (7, -1) circle (0.1cm);
\filldraw[fill=white] (7.5, -0.5) circle (0.1cm);
\filldraw[fill=white] (7.5, 0.5) circle (0.1cm);
\filldraw (8, 0) circle (0.1cm);
\filldraw[fill=white] (9, 0) circle (0.1cm);
\filldraw (10,0) circle (0.1cm);
\filldraw[fill=white] (11,0) circle (0.1cm);
\filldraw (12,0) circle (0.1cm); 
\filldraw[fill=white] (13,1) circle (0.1cm); 
\filldraw[fill=white] (13,-1) circle (0.1cm);
\filldraw[fill=white] (10,1) circle (0.1cm); 
\filldraw[fill=white] (6,0.5) circle (0.1cm); 
\filldraw[fill=white] (6,-0.5) circle (0.1cm); 
\filldraw[fill=white] (6,2) circle (0.1cm); 
\filldraw[fill=white] (6,-2) circle (0.1cm);

\node at (10, -2) {(C)};

\end{tikzpicture}
\end{center}
\label{fig:claw_graph}
\end{figure}

\end{example}

\begin{definition}
An \emph{induced claw unit} of a graph is a $K_{1,3}$ subgraph with $1$ vertex of degree $3$ in $G$ and $3$ vertices of degree less than or equal to $2$ in $G$ (See figure~\ref{fig:claw_unit}.)
\end{definition}

Notice that a clawed graph decomposes into induced claw units that pairwise intersect at 1 vertex. As we will see in Section~\ref{sec:contractible} not all graphs with maximum degree 3 can decompose into claw units in this manner. For example $WC^d_3$ the fully whiskered $3$-cycle .  

We will be interested in deleting an induced claw unit in a graph. To do so, we consider an induced claw unit $c$ to be defined by the unique degree $3$ vertex, call it $v$. We abuse notation and use $G \smallsetminus c$ to denote the vertex deleted subgraph of $G \smallsetminus \{v\}$, the graph obtained by deleting $v$ and all the edges incident to it. (See figure~\ref{fig:claw_unit}.)


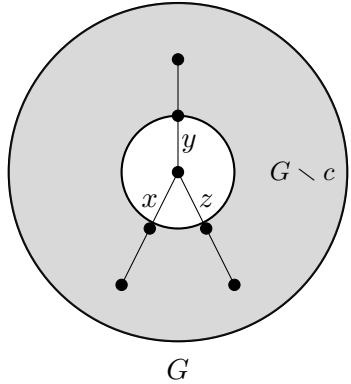
\begin{figure}[h]
\begin{center}
\begin{tikzpicture}[scale = .75]

\draw[thick] (0,0) circle (3);
\filldraw[thick, fill=gray, opacity=.3] (0,0) circle (3);

\filldraw[thick, fill=white] (0,0) circle (1);
\filldraw (0,0) circle (0.1cm);
\filldraw (1,-2) circle (0.1cm);
\filldraw (-1,-2) circle (0.1cm);
\filldraw (0,2) circle (0.1cm);
\filldraw (0,1) circle (0.1cm);
\filldraw (0.5,-1) circle (0.1cm);
\filldraw (-0.5,-1) circle (0.1cm);

\draw (0,2) -- (0,0) -- (1,-2);
\draw (0,0) -- (-1,-2);

\node at (2.2, 0) {\small $G \smallsetminus c$};
\node at (0.2, 0.5) {$y$};
\node at (0.5, -0.5) {$z$};
\node at (-0.5, -0.5) {$x$};
\node at (0, -3.5) {$G$};

\end{tikzpicture}
\end{center}
\caption{ The edge set $\{x,y,z\}$ defines an induced claw unit, call it $c$, of graph $G$. The shaded region is the graph $G \smallsetminus c$. }
\label{fig:claw_unit}
\end{figure}


\begin{definition}
A \emph{2-matching} of a graph $G$ is a subset of edges $H \subseteq E(G)$ such that any vertex $v \in V(H)$ has degree at most $2$. 
\end{definition}

\begin{definition}
A \emph{$2$-matching complex} of a graph $G$, denoted $M_2(G)$ is a simplicial complex with vertices given by edges of $G$ and faces given by $2$-matchings of $G$. 
\end{definition}

\begin{example}\label{ex:complexes}
See Figure~\ref{fig:complexes} consisting of the graph $G$, its matching complex $M_1(G)$, and its  $2$-matching complex $M_2(G)$. The $2$-matching complex of $G$ consists of 5 maximal faces. Namely, (1) $\{\textbf{a,c,d}\}$, (2) $\{\textbf{a,c,e}\}$, (3) $\{\textbf{a,b,d,e}\}$, (4)  $\{\textbf{b,c,d}\}$, (5) $\{\textbf{b,c,e}\}$.These maximal faces form a simplicial complex that is homotopy equivalent to $S^2$ a $2$-sphere. Notice that $M_1(G) \subseteq M_2(G)$.

\begin{figure}[!hbt]
\begin{center}
\begin{tikzpicture}[scale = .75]

\filldraw (-1, 1) circle (0.1cm);
\filldraw (-1, -1) circle (0.1cm);
\filldraw (2, 0) circle (0.1cm);
\filldraw (0, 0) circle (0.1cm);
\filldraw (3, 1) circle (0.1cm);
\filldraw (3, -1) circle (0.1cm);

\draw[thick, fill=gray, opacity=.8] (14,1)-- (13,0) -- (14,-1) -- (15,0) -- (14,1);
\filldraw (14, 1) circle (0.1cm);
\filldraw (13,0) circle (0.1cm);
\filldraw (14,-1) circle (0.1cm);
\filldraw (15,0) circle (0.1cm);
\filldraw (14, 3) circle (0.1cm);

\draw (13,0) -- (15,0);

\draw[thick, fill=gray, opacity=.5] (14,1)-- (14,3) -- (15,0) -- (14,1);

\draw[thick, fill=gray, opacity=.5] (14,1)-- (14,3) -- (13,0) -- (14,1);

\draw[thick, fill=gray, opacity=.3] (14,-1)-- (14,3) -- (15,0) -- (14,-1);

\draw[thick, fill=gray, opacity=.3] (14,-1)-- (14,3) -- (13,0) -- (14,-1);

\node[below] at (14, -2) {$M_2(G)$};

\node at (12.7, 0) {$a$};
\node at (15.3, 0) {$b$};
\node at (14.3, 1.15) {$d$};
\node at (14,3.3) {$c$};
\node at (14, -1.35) {$e$};

\draw (-1, 1) -- (0, 0) -- (2, 0) --(3, 1);
\draw (2, 0) -- (3, -1);
\draw (-1, -1) -- (0, 0) ;

\node[below] at (1, -2) {$G$};

\node at (-0.35,0.75) {$\textbf{a}$};
\node at (2.35,0.75) {$\textbf{d}$};
\node at (-0.35,-.75) {$\textbf{b}$};
\node at (1,0.25) {$\textbf{c}$};
\node at (2.35,-0.75) {$\textbf{e}$};


\filldraw (6, 0) circle (0.1cm);
\filldraw (7, 1) circle (0.1cm);
\filldraw (7, -1) circle (0.1cm);
\filldraw (8, 0) circle (0.1cm);
\filldraw (7, 3) circle (0.1cm);

\node at (5.7, 0) {$a$};
\node at (8.3, 0) {$b$};
\node at (7, 1.35) {$d$};
\node at (7, 3.3) {$c$};
\node at (7, -1.35) {$e$};

\draw (6, 0) -- (7, 1) -- (8, 0) --(7, -1) -- (6, 0);

\node[below] at (7, -2) {$M_1(G)$};

\end{tikzpicture}
\end{center}
\caption{}\label{fig:complexes}

\end{figure}

\end{example}


\subsection{Discrete Morse Theory} \label{Sec:DiscreteMorseTheory}
Robin Forman developed Discrete Morse Theory as a tool to study the homotopy type of simplicial complexes \cite{Forman}. The underlying idea of the theory is to pair faces in a simplicial complex to give rise to a sequence of collapses that yields a homotopy equivalent cell complex.

\begin{definition}\label{def:partialmatching}

A \textit{partial matching} in a poset $P$ is a partial matching in the underlying graph of the Hasse diagram
of $P$, i.e., it is a subset $M\subseteq P \times P$ such that
\begin{itemize}
\item
$(a,b)\in M$ implies $b \succ a;$ ,i.e., $a<b$ and no $c$ satisfies $a<c<b$.
\item
each $a\in P$ belongs to at most $1$ element in $M$.
\end{itemize}
When $(a,b) \in M$, we write $a=d(b)$ and $b=u(a)$.
\item
A partial matching on $P$ is called \emph{acyclic} if there does not exist a cycle
\[
a_1 \prec u(a_1) \succ a_2 \prec u(a_2) \succ \cdots \prec u(a_m) \succ a_1
\]
with $m\ge 2$ and all $a_i\in P$ being distinct.

\end{definition}

Given an acyclic partial matching $M$ on a poset $P$, an element $c$ is \emph{critical} if it is unmatched. If none of the critical cells can be further paired in the matching $M$ is called \emph{complete}. 
If every element is matched by $M$, $M$ is called \emph{perfect}.

The main theorem of discrete Morse theory as given in~\cite[Theorem 11.13]{Kozlov} is

\begin{theorem}\label{thm:main_mta}
Let $\Delta$ be a polyhedral cell complex and let $M$ be an acyclic matching on the face poset of $\Delta$.
Let $c_i$ denote the number of critical $i$-dimensional cells of $\Delta$.
The space $\Delta$ is homotopy equivalent to a cell complex $\Delta_c$ with $c_i$ cells of dimension $i$ for each $i\ge 0$, plus a single $0$-dimensional cell in the case where the empty set is paired in the matching.
\end{theorem}

A common way to obtain an acyclic matching is to \emph{toggle} on an element $x$ in the vertex set of a face poset $P$.
\begin{definition} 
Let $\mathcal{F}(\Delta)$ be the face poset of a simplicial complex $\Delta$ and $Q \subseteq \mathcal{F}(\Delta)$ a subposet. For $x$ an element in the vertex set of $\Delta$, \emph{toggling} on an element $x$ is a partial matching in $\mathcal{F}(\Delta)$ that pairs subsets $a \in Q$, $x \not\in a$ with $a \cup \{x\}$, whenever possible. We omit the subposet when clear from context. 
\end{definition}

It is often useful to create acyclic partial matchings on different sections of the face poset of a simplicial complex and then combine
them to form a larger acyclic partial matching on the entire poset.
This process is detailed in the following theorem known as the \textit{Cluster Lemma} in \cite{Jakob} and the \textit{Patchwork Theorem} in \cite{Kozlov}.

\begin{theorem}\label{thm:patchwork}
Assume that $\varphi : P \rightarrow Q$ is an order-preserving map.
For any collection of acyclic matchings on the subposets $\varphi^{-1}(q)$ for $q\in Q$, the union of these matchings is itself an acyclic matching on $P$.
\end{theorem}

The following theorem shows there is an intimate relationship between linear extensions and acyclic matchings \cite{Kozlov}. 
 
\begin{theorem}[Kozlov, Theorem 11.2] \label{thm:linear_extension}
A partial matching on a poset $P$ is acyclic if and only if there exists a linear extension of $\mathcal{L}$ of $P$ such that $x$ and $u(x)$ follow consecutively. 
\end{theorem} 

Since $x$ and $u(x)$ follow consecutively in the linear extension, when we refer to these elements in the linear extension we will use the notation $(x,u(x))$ and consider them as a pair of consecutive elements in the poset. 

\begin{lemma}\label{lem:toggle}
  Let $Q \subseteq \mathcal{F}(\Delta)$ be a subposet. Toggling on $Q$ provides an acyclic partial matching. 
    \end{lemma}
    \begin{proof}
    To see this suppose we toggle on the element $x$. Start with an element $a_1 \in P$ such that $x \not\in a_1$, $x \in u(a_1)$. Any element $a_2 \prec u(a_1)$ with $a_2 \neq a_1$ contains $x$ since $(a_1, u(a_1)) \in M$. Hence, there is no element $u(a_2)$, and a cycle cannot be created. 

    \end{proof}

Additionally, using the patchwork theorem, we see that performing repeated toggling yields an acyclic partial matching. 

\begin{lemma}
Let $\mathcal{F}(\Delta)$ be the face poset of a simplicial complex $\Delta$ and suppose $x_1, x_2, \dots, x_n$ is a sequence of vertices if $\Delta$. Repeatedly toggling on $x_1$, then $x_2$, and so on, in $\mathcal{F}(\Delta)$ yields an acyclic partial matching on $\mathcal{F}(\Delta)$. 
\end{lemma}

\begin{proof}
Let $Q$ be a poset with elements $\mathcal{R}$ and $Y_{i}$ for $i \in [n]$ with relations given by $Y_i \prec Y_{i+1}$ for all $i \in [n-1]$ and $Y_n \prec \mathcal{R}$. Let $D_1 = \{\alpha \in P | x_1 \in \alpha \text{ or } \alpha \cup \{x_1\} \in P\}$ and recursively define $D_i := \{\alpha \in P | x_i \in \alpha \text{ or } \alpha \cup \{x_i\} \in P, \text{ and } \alpha \not\in D_j \text{ for } j \leq i-1\}$ for $2\leq i \leq n$. Define $\varphi: \mathcal{F}(\Delta) \rightarrow Q$ by  $\varphi^{-1}(Y_i)  := D_i$, for $1 \leq i \leq n$, and the remaining elements to $\mathcal{R}$. The map $\varphi$ is well-defined and order-preserving. On $\varphi^{-1}(Y_{i})$ toggle on $x_i$, which is an acyclic matching by Lemma \ref{lem:toggle}. The union of these preimages forms an acyclic matching on $\mathcal{F}(\Delta)$ by Theorem~\ref{thm:patchwork}.  
\end{proof}

We will use discrete Morse theory to determine the homotopy type of clawed graphs. We observe now that induced claw units in graphs behave nicely with $2$-matching complexes.

\begin{proposition} \label{prop:bijections}
Let $c \in G$ be an induced claw unit with edge set $E(c) = \{x,y,z\}$. 
The following sets are in bijection with each other: 
\begin{enumerate} 
\item [(i)] The set of $2$-matchings of $G \smallsetminus c$,
\item [(ii)] The set of $2$-matchings containing $\{y,z\}$, and
\item [(iii)] The set of $2$-matchings containing $x$ and not $y$ or $z$.
\end{enumerate}
\end{proposition}

\begin{proof}
For any $2$-matching $m$ in $G \smallsetminus c$, both $m \cup \{x\}$ (not containing $y$ or $z$) and $m \cup \{y,z\}$ are $2$-matchings in $G$. Notice that $x$ and $\{y,z\}$ cannot be in a $2$-matching together since they all meet at a degree three vertex. 
\end{proof}

\begin{example}
Consider the graph in Figure~\ref{fig:paw}. There is exactly $1$ induced claw unit, call it $c$, given by the edge set $\{x,y,z\}$. The set $\{e\}$ is the only $2$-matching of $G \smallsetminus c$. Notice that the $2$-matchings containing $\{y,z\}$ consists of exactly $\{e,y,z\}$ and $2$-matchings containing $x$ and not $y$ or $z$ consists of exactly $\{e,x\}$.  
\end{example}

\begin{figure}[h]
\begin{center}
\begin{tikzpicture}[scale = .75]

\filldraw (2,2.5) circle (0.1cm);
\filldraw (2,4) circle (0.1cm);
\filldraw (1,5) circle (0.1cm);
\filldraw (3,5) circle (0.1cm);

\draw (2,4) -- (1,5) -- (3,5) --(2,4) -- (2,2.5);

\node at (2, 5.3) {$e$};
\node at (1.2, 4.3) {$y$};
\node at (2.8, 4.3) {$z$};
\node at (2.3, 3.3) {$x$};

\end{tikzpicture}
\end{center}
\caption{}
\label{fig:paw}
\end{figure}


We turn our attention to a general connectivity result of $M_2(G)$ for any graph $G$. Since $\mathcal{F}(M_2(G))$ the face poset of a $2$-matching complex of $G$ has vertex set consisting of faces of $M_2(G)$ with an order relation of containment, for $a,b \in \mathcal{F}(M_2(G))$, $a \prec b$ if $b = a \cup e$ for some $e \in E(G)$. In relation to Figure~\ref{fig:paw}, suppose we define a partial matching on $\mathcal{F}(M_2(G))$ by toggling on $x$, where $x \in E(G)$. Then, the matchings remaining after toggling are exactly those that contain $\{y,z\}$ and therefore are in bijection with $2$-matchings of $G \smallsetminus c$ by Proposition~\ref{prop:bijections}. Hence, if you have $2$ induced claw units $c_1$ and $c_2$ in $G$, the choice of toggle edge in $c_1$ and $c_2$ and the order in which one toggles is irrelevant.

\begin{lemma} \label{lem:claw}
Let $G$ be a simple, finite graph and $\mathcal{C} = \{c_1,...,c_n\}$ be a collection of induced claw units in $G$ with $E(c_i) := \{x_i, y_i, z_i\}$ for each $c_i$. Then the connectivity of $M_2(G)$ is at least $2|\mathcal{C}|-2$ and $M_2(G) \simeq S^{2|\mathcal{C}| -1} \ast M_2(G \smallsetminus \mathcal{C})$. Further, if we fix the toggle edge in each $c_i$, say $x_i$, then every critical cell remaining after toggling on all of the $x_i$'s will consist of $\{y_i,z_i\}$ for all $i$, regardless of order.  
\end{lemma}

\begin{proof}
To see that $M_2(G) \simeq S^{2|\mathcal{C}| -1} \ast M_2(G \smallsetminus \mathcal{C})$, notice that for all $i$, $M_2(c_i) = M_2(K_{1,3}) \simeq S^1$ is the boundary of a $2$-simplex and any $2$-matching of $G$ is the union of a $2$-matching from $c_i$, for each $i$, and a $2$-matching of $G \smallsetminus \mathcal{C}$. Therefore,
\[ M_2(G) \simeq M_2(c_1) \ast M_2(c_2) \ast \cdots \ast M_2(c_n) \ast M_2(G \smallsetminus \mathcal{C}) \simeq S^{2|\mathcal{C}| -1} \ast M_2(G \smallsetminus \mathcal{C}).
\]  
It follows that the connectivity of $M_2(G)$ is at least $2|\mathcal{C}|-2$. 

 To see that the remaining critical cell will consist of $\{y_i,z_i\}$ for all $i$, regardless of order, let $P := \mathcal{F}(M_2(G))$ be the face poset of the $2$-matching complex of $G$. We define a partial (discrete Morse) matching on $P$ by (arbitrarily) fixing $x_i$ as the toggle edge for each $c_i$. Our claim is that for any permutation $\pi \in \mathfrak{S}_n$, the unmatched subposet that remains after toggling on $x_{\pi(1)}, x_{\pi(2)}, x_{\pi(3)}, \dots, x_{\pi(n)}$ is the upper-order ideal $P_{\geq \{y_{\pi(1)}, z_{\pi(1)}, y_{\pi(2)}, z_{\pi(2)}, \dots ,y_{\pi(n)}, z_{\pi(n)}\}}$. Since permutations can be generated by a sequence of transpositions, it suffices to consider the unmatched subposet obtained from toggling $x_1$, then $x_2$ and the unmatched subposet obtained from toggling $x_2$, then $x_1$. 

Suppose first that we toggle on $x_1$. The edge $x_1 \in E(G)$ forms a $2$-matching with all $2$-matchings of $G$ that do not contain both $y_1$ and $z_1$ so the unmatched cells of $P$ are precisely the elements containing both $y_1$ and $z_1$. That is, the unmatched subposet that remains is $P_{\geq \{y_1,z_1\}}$. Now, toggling on $x_2$ matches all of the $2$-matchings of $G$ that contain $y_1,z_1$, but do not contain $y_2,z_2$. All elements $b \in P_{\geq\{y_1,z_1\}}$ with $\{x_2\} \in b$ will be paired with $a := b \smallsetminus \{x_2\}$ through toggling on $x_2$ and all elements $a$ are in $ P_{\geq\{y_1,z_1\}}$ since $\{y_1,z_1\} \in b$. Notice that all matchings in $P_{ \geq \{y_1,z_1\}}$ are in bijection with $2$-matchings in $G \smallsetminus c_1$  by Proposition~\ref{prop:bijections} and $c_2 \in G \smallsetminus c_1$. 

Hence, the unmatched subposet that remains after toggling on $x_1$ then $x_2$ is precisely $P_{\geq \{y_1,z_1,y_2,z_2\}}$. An analogous argument shows that the same upper order ideal remains after toggling first on $x_2$ and then $x_1$.
By induction, we get that the unmatched subposet that remains after toggling on $x_{\pi(1)}, x_{\pi(2)}, x_{\pi(3)}, \dots x_{\pi(n)}$ is the upper-order ideal $P_{\geq \{y_{\pi(1)}, z_{\pi(1)}, y_{\pi(2)}, z_{\pi(2)}, \dots ,y_{\pi(n)}, z_{\pi(n)}\}}$. 

\end{proof}

\begin{definition}
Let $G$ be any graph. A \emph{claw-induced partial matching} is an acyclic partial matching on $\mathcal{F}(M_2(G))$ obtained by toggling on elements in the vertex set of  $\mathcal{F}(M_2(G))$ corresponding to edges in induced claw units of $G$, whenever possible.  
\end{definition}

\subsection{Matching Tree Algorithm (MTA)}

In \cite{Melou_Linusson_Nevo}, the authors detail the Matching Tree Algorithm which provides an acyclic discrete Morse matching on the face poset of an independence complex of a graph $G$. 
An \emph{independence complex} $Ind(G)$ of a graph $G$ is a simplicial complex in which the vertices are given by vertices of $G$ and faces are given by independent sets of vertices. The matching complex of a graph $G$ is equal to the independence complex of the line graph of $G$ where the vertices of the line graph are the edges of the graph and $2$ vertices are adjacent if and only if the corresponding edges are incident in the graph. 
In Section~\ref{sec:wheel}, we will use the Matching Tree Algorithm to find the homotopy type of the $1$-matching complex of a wheel graph, by looking at the independence complex of the line graph. 

 Let $G$ be a simple graph with vertex set $V=V(G)$. Bousquet-M\'elou, Linusson, and Nevo motivate the MTA with the following algorithm. Let $\Sigma$ denote the independence complex of $G$. Take a vertex $p \in V$ and denote $N(p)$ as the set of its neighbors. Define $\Delta = \{I \in \Sigma: I \cap N(p)=\emptyset\}$. For  $I \in \Delta$ and $p \not\in I$, the set of pairs $(I, I \cup \{p\})$ form a perfect matching of $\Delta$ and hence a partial matching of $\Sigma$. The vertex $p$ is called a \emph{pivot}. 
 
Notice that the unmatched elements of $\Sigma$ are those containing at least $1$ element of $N(p)$. Choose an unmatched vertex and continue the process as many times as possible. This algorithm will give rise to a rooted tree, called a \emph{matching tree of $\Sigma$}, whose nodes represent sets of unmatched elements. Some of the nodes are reduced to the empty set, and all others are of the form 
 \[
 \Sigma(A,B) = \{ I \in \Sigma: A \subseteq I \text{ and } B \cap I = \emptyset\},
 \]
 where 
 \[
 A \cap B = \emptyset \text{ and } N (A) := \bigcup\limits_{a \in A} N(a) \subseteq B. 
 \]
 
 The root of the tree is $\Sigma(\emptyset, \emptyset)$, which is equal to the set of all the independent sets of $G$. As we traverse the tree the sets $\Sigma(A,B)$ will become smaller and the leaves of the tree will have cardinality $0$ or $1$. 

The following presentation of the Matching Tree Algorithm follows \cite{JelicEtAl}. 
Begin with the root node $\Sigma(\emptyset, \emptyset)$ and at each node $\Sigma(A,B)$ where $A \cup B \neq V$ apply the following procedure: 
\begin{enumerate}
\item[(1)] If there is a vertex $v \in V \smallsetminus (A \cup B)$ such that $N(v) \smallsetminus (A \cup B) = \emptyset$, then $v$ is called a \emph{free vertex}. Give $\Sigma(A, B)$ a single child labeled $\emptyset$. 
\item[(2)] Otherwise, if there is a vertex $v \in V \smallsetminus (A \cup B)$ such that $N(v) \smallsetminus (A \cup B)$ is a single vertex $w$, then $v$ is called a \emph{pivot} and $w$ a \emph{matching vertex}. Give $\Sigma(A,B)$ a single child labeled $\Sigma(A \cup \{w\}, B \cup N(w))$. 
\item[(3)] When there is no vertex that satisfies $(1)$ or $(2)$ and $A \cup B \neq V$, choose a \emph{tentative pivot} in $V' = V \smallsetminus (A \cup B)$ and give $\Sigma(A,B)$ $2$ children $\Sigma(A \cup \{v\}, B \cup N(v))$, which we call the\emph{ right child}, and $\Sigma(A, B \cup \{v\})$, which we call the \emph{left child}.  
\end{enumerate} 

\begin{remark} Step (3) is motivated by the observation that if $v$ has at least $2$ neighbors, say $w$ and $w'$ then some of the unmatched sets $I$ contain $w$, and some others don't, but if they do not contain $w$ than they must contain $w'$. 
\end{remark}

The following theorem is the main theorem for the Matching Tree Algorithm, which is due to Bousquet-M\'elou, Linusson, and Nevo \cite{Melou_Linusson_Nevo}, but is stated as it appears in Braun and Hough \cite{Braun_Hough}. 

\begin{theorem}
A matching tree for $G$ yields an acyclic partial matching on the face poset of $Ind(G)$ whose critical cells are given by the non-empty sets $\Sigma(A,B)$ labeling non-root leaves of the matching tree. In particular, for each such set $\Sigma(A,B)$, the set $A$ yields a critical cell in Ind($G$). 
\end{theorem}

 Thus far, we have provided combinatorial tools for determining the homotopy type of simplicial complexes. It is also possible to use more topological methods to approach homotopy type. This approach requires inductively determining the homotopy type of complexes of interest and appropriately ``gluing" these spaces over a common subspace. For a more detailed discussion see \cite[Section 4.G]{Hatcher} and \cite[Section 15.2]{Kozlov}.
The following lemma follows from \cite[Proposition 4G.1]{Hatcher}, where $X \vee Y$ is considered as a homotopy colimit. 

\begin{lemma}\label{lem:colim}
If $X$ and $Y$ are $2$ spaces such that $X \simeq_f X'$ and $Y \simeq_g Y'$, then $X \vee Y \simeq X' \vee Y'$. 
\end{lemma}


\section{Contractibility in $2$-matching complexes} \label{sec:contractible}

We begin this section by exploring graph properties that force contractible  $2$-matching complexes. 


\begin{observation} \label{obs:wedge}
 If $G$ and $H$ are $2$ disjoint graphs with leaves $v_1 \in V(G)$ and $v_2 \in V(H)$, it is immediate for $G \underset{v_1 \sim v_2}{\vee} H$ we have
\[ M_2(G \underset{v_1 \sim v_2}{\vee} H) = M_2(G) \ast M_2(H), \]
where $\ast$ denotes the topological join. 

\end{observation}

Recall that $G_{(v_1, v_2)}$ is the graph obtained by identifying the vertices $v_1$ and $v_2$ in $G$.  

\begin{proposition} \label{prop:leafid}
Let $G$ be a graph with at least $3$ vertices. If $v_1, v_2 \in V(G)$ are $2$ leaves of a graph $G$, then $M_2(G) =  M_2(G_{(v_1,v_2)})$. 
\end{proposition}

\begin{proof}
Let $G$ be a graph with leaves $v_1, v_2$ and $H \subseteq E(G)$ a $2$-matching. Since all vertices $v \in V(H)$ have degree at most $2$, identifying the $2$ leaves $v_1, v_2$ does not affect $H$. So, $H \subseteq E(G) = E(G_{(v_1,v_2)})$ is also a $2$-matching of $G_{(v_1,v_2)}$. 
\end{proof}

\begin{theorem}\label{thm:bridge}
Let $G$ be a graph with $e=\{x,y\} \in E(G)$ such that deg$(x) \leq 2$ and deg$(y) \leq 2$. Then $M_2(G)$ is contractible. 
\end{theorem}

\begin{proof}
Since both endpoints of $e$ have degree at most $2$, $e$ may be included in any $2$-matching of $G \smallsetminus e$ and $M_2(G) \simeq e \ast M_2(G \smallsetminus e)$. Hence, $M_2(G)$ is a cone and therefore contractible. 
\end{proof}

Theorem~\ref{thm:bridge} gives rise to a large class of graphs that have contractible $2$-matching complexes. We will now explore $2$-matching complexes that are close to but not contractible. In particular, we turn our attention to clawed graphs. We begin by considering clawed paths. In the following proposition, we use the well-known fact that for $2$ spheres $S^m$ and $S^n$, $S^m \ast S^n \simeq S^{m+n+1}$.

\begin{proposition} \label{prop:clawed_paths}
For $n \geq 0$, let $CP_n$ be a clawed path with respect to a path of length $n$. Then, $M_2(CP_n) \simeq S^{2n+1}$.  
\end{proposition}
\begin{proof}
Since $P_0$ consists of $1$ vertex and no edges, we have $CP_0 = K_{3,1}$. See Figure~\ref{fig:CP01}. It follows that $M_2(CP_0) \simeq S^1$. Consider now a clawed path of length $1$, $CP_1$ consists of $2$ copies of $K_{3,1}$ intersecting at $1$ vertex. By Observation~\ref{obs:wedge} we have $M_2(CP_1) = M_2(CP_0 \vee CP_0) = M_2(CP_0) \ast M_2(CP_0) = S^1 \ast S^1 \simeq S^{1+1+1} = S^3$. Continuing inductively, we have $M_2(CP_n) = M_2(CP_{n-1} \vee CP_0) \simeq S^{2(n-1) +1} \ast S^1 \simeq S^{2n-2 +3} = S^{2n+1}$.  
\end{proof} 

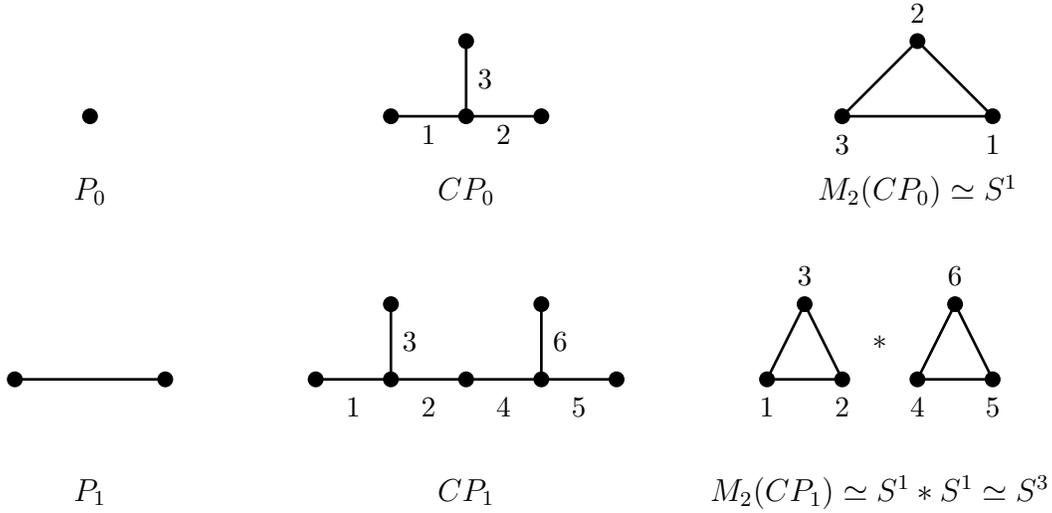
\begin{figure}
\begin{center}
\begin{tikzpicture}[scale = 0.5, every node/.style={}]

\fill[] (-10,5) circle (6pt);
\node at (-10,3) {\large{$P_0$}};

\node at (0,3) {\large{$CP_0$}};
\fill[] (-2,5) circle (6pt);
\fill[] (0,5) circle (6pt);
\fill[] (2,5) circle (6pt);
\fill[] (0,7) circle (6pt);

\draw[line width=1pt] (0,5) --
(-2,5);
\draw[line width=1pt] (0,5) --
(2,5);
\draw[line width=1pt] (0,5) --
(0,7);

\node at (-1,4.5) {$1$};
\node at (1,4.5) {$2$};
\node at (0.5,6) {$3$};

\node at (12,3) {\large{$M_2(CP_0) \simeq S^1$}};
\fill[] (10,5) circle (6pt);
\fill[] (14,5) circle (6pt);
\fill[] (12,7) circle (6pt);

\node at (10,4.25) {$3$};
\node at (12,7.75) {$2$};
\node at (14,4.25) {$1$};

\draw[line width=1pt] (10,5) --
(14,5);
\draw[line width=1pt] (10,5) --
(12,7);
\draw[line width=1pt] (12,7) --
(14,5);

\node at (-10,-5) {\large{$P_1$}};
\fill[] (-12,-2) circle (6pt);
\fill[] (-8,-2) circle (6pt);
\draw[line width=1pt] (-12,-2) --
(-8,-2);

\node at (0,-5) {\large{$CP_1$}};
\draw[line width=1pt] (-4,-2) --
(4,-2);
\fill[] (0,-2) circle (6pt);
\fill[] (-2,-2) circle (6pt);
\fill[] (2,-2) circle (6pt);
\fill[] (-4,-2) circle (6pt);
\fill[] (4,-2) circle (6pt);
\fill[] (2,0) circle (6pt);
\fill[] (-2,0) circle (6pt);
\draw[line width=1pt] (-2,-2) --
(-2,0);
\draw[line width=1pt] (2,-2) --
(2,0);

\node at (-3,-2.75) {$1$};
\node at (-1,-2.75) {$2$};
\node at (1,-2.75) {$4$};
\node at (3,-2.75) {$5$};
\node at (-1.5,-1) {$3$};
\node at (2.5,-1) {$6$};

\node at (11,-5) {\large{$M_2(CP_1) \simeq S^1 \ast S^1 \simeq S^3$}};
\fill[] (8,-2) circle (6pt);
\fill[] (10,-2) circle (6pt);
\fill[] (9,0) circle (6pt);
\draw[line width=1pt] (8,-2) --
(10,-2) -- (9,0) -- (8,-2);

\fill[] (12,-2) circle (6pt);
\fill[] (14,-2) circle (6pt);
\fill[] (13,0) circle (6pt);
\draw[line width=1pt] (12,-2) --
(14,-2) -- (13,0) -- (12,-2);

\node at (11,-1) {$\ast$};

\node at (8,-2.75) {$1$};
\node at (10,-2.75) {$2$};
\node at (9,0.75) {$3$};
\node at (12,-2.75) {$4$};
\node at (14,-2.75) {$5$};
\node at (13,0.75) {$6$};

\end{tikzpicture}
\caption{$M_2(CP_0)$ and $M_2(CP_1)$ as in proof of Proposition~\ref{prop:clawed_paths}}
\label{fig:CP01}
\end{center}
\end{figure}


\begin{corollary} 
$M_2(CP_{n-1}) \simeq M_2(CC_{n}) \simeq S^{2n-1}.$
\end{corollary} 
\begin{proof}
The result follows from Proposition~\ref{prop:leafid}; see Figure~\ref{fig:path_to_cycle}.
\end{proof} 
In the next proposition, we see that, even further, the $2$-matching complex for a clawed cycle shares its homotopy type with the $2$-matching complex of a fully whiskered cycle. Although there are no induced claws in a fully whiskered cycle, we will see that, in this case, we can treat induced $K_{3,1}$ subgraphs in a similar manner.

\begin{definition}
A \emph{fully whiskered graph} $WG$ is a graph in which a leaf is attached to every vertex of the graph $G$. 
\end{definition}

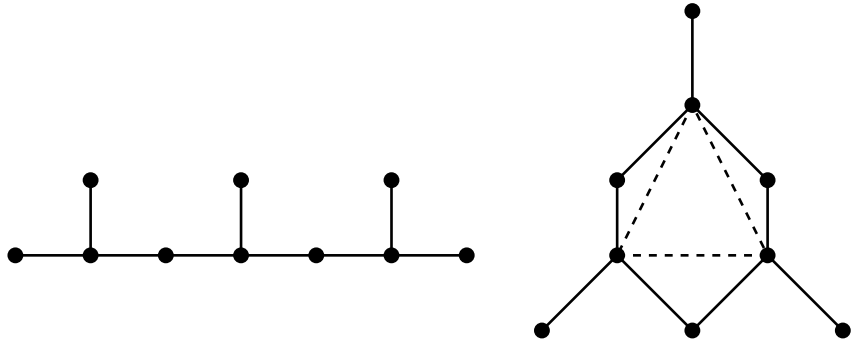
\begin{figure}[!hbt]
\begin{center}
\begin{tikzpicture}[scale = 0.5, every node/.style={}]


\draw[line width=1pt] (-4,-2) --
(4,-2);
\fill[] (0,-2) circle (6pt);
\fill[] (-2,-2) circle (6pt);
\fill[] (2,-2) circle (6pt);
\fill[] (-4,-2) circle (6pt);
\fill[] (4,-2) circle (6pt);
\fill[] (2,0) circle (6pt);
\fill[] (-2,0) circle (6pt);
\draw[line width=1pt] (-2,-2) --
(-2,0);
\draw[line width=1pt] (2,-2) --
(2,0);

\fill[] (-6,-2) circle (6pt);
\fill[] (-8,-2) circle (6pt);
\fill[] (-6,0) circle (6pt);
\draw[line width=1pt] (-6,-2) --
(-8,-2);
\draw[line width=1pt] (-4,-2) --
(-6,-2);
\draw[line width=1pt] (-6,0) --
(-6,-2);

\fill[] (12,0) circle (6pt);
\fill[] (12,-2) circle (6pt);
\fill[] (10,-4) circle (6pt);
\fill[] (8,-2) circle (6pt);
\fill[] (8,0) circle (6pt);
\fill[] (10,2) circle (6pt);

\draw[line width=1pt] (12,0) --
(12,-2) -- (10,-4) -- (8,-2) -- (8,0) -- (10, 2) -- (12,0);

\fill[] (6,-4) circle (6pt);
\fill[] (14,-4) circle (6pt);
\fill[] (10,4.5) circle (6pt);

\draw[line width=1pt] (8,-2) --
(6,-4);
\draw[line width=1pt] (12,-2) --
(14,-4);
\draw[line width=1pt] (10,4.5) --
(10,2);

\draw[dashed,line width=1pt] (8,-2) --
(12,-2) -- (10,2) -- (8,-2);

\end{tikzpicture}
\caption{On the left graph $CP_2$, the clawed path of length $2$ and on the right $CC_3$, the clawed $3$-cycle obtained by identifying the endpoints of $CP_2$. The core $3$-cycle is shown with dashed lines.}
\label{fig:path_to_cycle}
\end{center}

\end{figure}

\begin{proposition} \label{prop:full_even_cycle}
Let $WC_m$ denote a fully whiskered $2m$-cycle graph for $m \geq 3$. $M_2(WC_{m}) \simeq S^{2m-1}$. 
\end{proposition}

\begin{proof}
Label the edges of the cycle by $1,2,...,2m$ and each leaf edge by $x_{i,i+1}$ for  $i \in [2m-1]$, and $x_{1,2m}$, where the index corresponds to the incident edges in the cycle as in Figure~\ref{fig:discrete_morse_cycles}. Let the edge set $c_i := \{x_{i, i+1}, i, i+1\}$ for each $i \in \{ 1,3,5, \dots ,2m-1\}$ denote an induced $K_{3,1}$ subgraph of $WC_m$. Then the collection $\mathcal{C} = \{c_1,c_3, \dots, c_{2m-1} \}$ of subgraphs defines a family of $m$ induced  induced $K_{3,1}$ subgraphs that are edge disjoint. If this were not the case, then there would exist an edge $j \in E(WC_m)$ that would be an edge in $2$ claws, but by the labeling system this would mean that $j = j+1$ which is a contradiction to the edge labels on the cycle. Notice that the collection $\mathcal{C}$ does not fully partition the graph, nonetheless performing a toggle on $1$ edge of each  induced $K_{3,1}$ subgraph gives rise to a complete matching. 
Following the proof of Lemma~\ref{lem:claw}, for each $i \in \{ 1,3,5, \dots ,2m-1\}$ let $x_{i,i+1}$ be the toggle edge in the discrete Morse matching on the face poset of $M_2(WC_m)$. It follows that
$M_2(WC_m) \simeq S^{2m -1} *M_2(WC_m \smallsetminus \mathcal{C})$ and since $M_2(WC_m \smallsetminus \mathcal{C})$ is contractible, the result follows. The only critical cell is $\{1,2, \dots, 2m\}$.
\end{proof}

\begin{corollary}\label{cor:sim}
$M_2(WC_n) \simeq M_2(CC_n)$ for $n\geq3$.
\end{corollary}

In Proposition~\ref{prop:full_even_cycle}, we considered fully whiskered $2m$-cycle graphs because we are interested in aligning this result with clawed path graphs, but there is no reason why we could not apply the same reasoning for fully whiskered odd-cycle graphs.

\begin{theorem}\label{rmk:full_odd_cycle}
Let $WC_n^d$ denote a fully whiskered $n$-cycle graph for odd $n$. Then, $M_2(WC_n^d) \simeq S^{n-1}$. 
\end{theorem}

\begin{proof}
Using the same $K_{3,1}$-induced partial matching as in the proof of Proposition~\ref{prop:full_even_cycle} for all $i \in \{1,3, \dots,n-2\}$, the remaining unmatched cells must contain $\{1,2, \dots, n-1\}$. These cells form an upper order ideal in the partially matched face poset of $M_2(WC_n^d)$ and include precisely $\{x_{n,1}, 1,2,...,n-1,x_{n-1,n}\}, \{1,2,...,n\}, \{x_{n,1}, 1,2,...,n-1\}, \{1,2,...,n-1,x_{n-1,n}\}$, and $\{1,2, ..., n-1\}$. Performing a final toggle on the edge $x_{n-1,n}$, we obtain $1$ critical cell, $\{1,2,...,n\}$ and hence $M_2(WC_n^d) \simeq S^{n-1}$. 
\end{proof}


\begin{figure}[!hbt]
\begin{center}
\begin{tikzpicture}[scale = 0.5, every node/.style={}]

\draw[double, line width=2pt] (2,0) --
(2,-2) -- (0,-4) -- (-2,-2) -- (-2,0) -- (0, 2) -- (2,0);
\fill[] (2,0) circle (6pt);
\fill[] (2,-2) circle (6pt);
\fill[] (0,-4) circle (6pt);
\fill[] (-2,-2) circle (6pt);
\fill[] (-2,0) circle (6pt);
\fill[] (0,2) circle (6pt);

\fill[] (-4,-4) circle (6pt);
\fill[] (4,-4) circle (6pt);
\fill[] (0,4.5) circle (6pt);
\fill[] (0,-6.5) circle (6pt);
\fill[] (-4,3) circle (6pt);
\fill[] (4,3) circle (6pt);

\draw[line width=1pt] (-2,-2) --
(-4,-4);
\draw[line width=1pt] (2,-2) --
(4,-4);
\draw[line width=1pt] (0,4.5) --
(0,2);
\draw[line width=1pt] (0,-6.5) --
(0,-4);
\draw[line width=1pt] (-4,3) --
(-2,0);
\draw[line width=1pt] (4,3) --
(2,0);

\node at (0.5,0.75) {$1$};
\node at (1.5,-1) {$2$};
\node at (0.75,-2.5) {$3$};
\node at (-0.75,-2.5) {$4$};
\node at (-1.5,-1) {$5$};
\node at (-0.75,0.65) {$6$};

\node at (0.75,3) {$x_{1,6}$};
\node at (3.5,0.75) {$x_{1,2}$};
\node at (2.5,-3.5) {$x_{2,3}$};
\node at (-0.75,-5) {$x_{3,4}$};
\node at (-2.5,-3.5) {$x_{4,5}$};
\node at (-3.5,0.75) {$x_{5,6}$};


\draw[double,line width=2pt] (10, 2) -- (12,0) -- (12,-2) -- (8,-2) -- (8,0);

\draw[line width=1pt] (8,0) -- (10, 2);
\fill[] (12,0) circle (6pt);
\fill[] (12,-2) circle (6pt);
\fill[] (8,-2) circle (6pt);
\fill[] (8,0) circle (6pt);
\fill[] (10,2) circle (6pt);

\fill[] (6,-4) circle (6pt);
\fill[] (14,-4) circle (6pt);
\fill[] (10,4.5) circle (6pt);
\fill[] (6,2) circle (6pt);
\fill[] (14,2) circle (6pt);

\draw[line width=1pt] (8,-2) --
(6,-4);
\draw[line width=1pt] (12,-2) --
(14,-4);
\draw[line width=1pt] (10,4.5) --
(10,2);
\draw[line width=1pt] (6,2) --
(8,0);
\draw[line width=1pt] (14,2) --
(12,0);

\node at (10.5,0.75) {$1$};
\node at (11.5,-1) {$2$};
\node at (10,-1.5) {$3$};
\node at (8.5,-1) {$4$};
\node at (9.25,0.75) {$5$};

\node at (10.75,3) {$x_{1,5}$};
\node at (13.5,0.75) {$x_{1,2}$};
\node at (12.5,-3.5) {$x_{2,3}$};
\node at (7.5,-3.5) {$x_{3,4}$};
\node at (6.5,0.75) {$x_{4,5}$};

\end{tikzpicture}
\caption{On the left a complete matching on $WC_3$ as in Proposition~\ref{prop:full_even_cycle} with edges in the critical cell highlighted with a double line and on the left the partial matching on $WC_5^d$ as in Remark~\ref{rmk:full_odd_cycle} with edges in the critical cell highlighted with a double line.}
\label{fig:discrete_morse_cycles}
\end{center}

\end{figure}
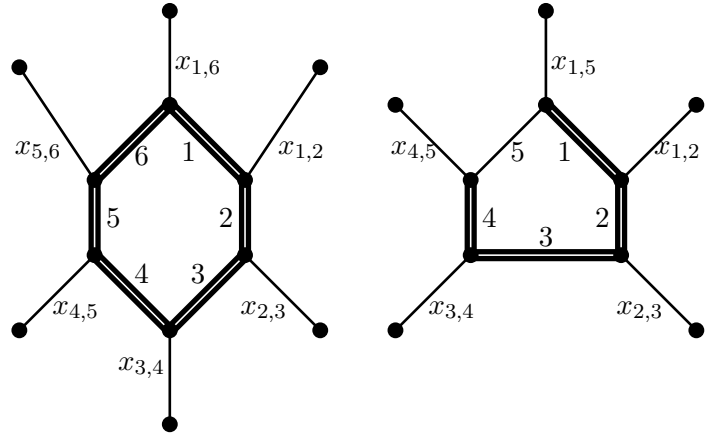




 We saw in Corollary~\ref{cor:sim} that $M_2(CC_n) \simeq M_2(WC_n) \simeq S^{2n-1}$ and it is no coincidence that $CC_n$ is a subgraph of $WC_n$. The next lemma shows that there are certain degree $2$ vertices such that attaching a leaf does not affect the homotopy type of the $2$-matching complex. We call such vertices \emph{attaching sites}.

\begin{lemma} \label{lem:clawed}
Let $CG$ be a clawed graph with vertex set $V(CG)$, edge set $E(CG)$, and $v \in V(CG)$ a degree $2$ vertex with $e_1, e_2 \in E(CG)$ the $2$ incident edges to $v$. Consider a complete claw-induced partial matching on $P$, the face poset of $M_2(CG)$. Then both edges $e_1$ and $e_2$ are in a critical cell if and only if attaching a leaf to $v$ does not change the homotopy type. Further, if at least $1$ edges from the set $\{e_1, e_2\}$  is not in any critical cell obtained from the complete claw-induced partial matching of $P$, the $2$-matching complex of $CG$ with a leaf attached to $v$ is contractible. 
\end{lemma}

\begin{proof}
Since $CG$ is a clawed graph and deg$(v) = 2$, $v$ is the intersection of $2$ claws $c_1$ and $c_2$. For each claw, $1$ of the edges is a toggle edge and $2$ are in critical cells. 
If $e_1$ and $e_2$ are in some critical cell; then they are in all critical cells since this would mean that $1$ of the other edges in $c_1$ and $c_2$ are toggled on. 
In this case, attaching a leaf $w$ to $v$ does not give rise to any additional critical cells since this would imply that $e_1, e_2$, and the edge $\{v,w\}$ are all in a $2$-matching together, but this is not possible because they are all incident a common vertex. 

Suppose now that no critical cell contains both $e_1$ and $e_2$ (but perhaps contains $1$). Then attaching a leaf $w$ to $v$ gives rise to several new critical cells, under the same matching $\mathcal{M}$. For each critical cell $X$ in the claw-induced partial matching on $P$, $X \cup \{w,v\}$ is a critical cell in the claw-induced partial matching on $\mathcal{F}(M_2(CG \cup \{w,v\}))$. Therefore, every critical cell can be further matched by toggling on $\{w,v\}$ and $M_2(CG \cup \{w,v\})$ is contractible. 
\end{proof}

\begin{theorem} \label{thm:homotopyType}
For a clawed graph $CG$, $M_2(CG) \simeq S^{\frac{2}{3}n -1}$ where $n = |E(CG)|$. 
\end{theorem}

\begin{proof}
The clawed graph $CG$ consists of  $\mathcal{C}$ a collection of claws that have pairwise intersection of at most $1$ vertex, that is a collection of $\frac{1}{3} n$ induced claw units, which fully partitions $G$. By Lemma~\ref{lem:claw}, $M_2(CG)\simeq S^{\frac{2}{3}n -1} *M_2(CG \smallsetminus \mathcal{C})$. Since $CG$ is fully partitioned by $\mathcal{C}$, the result follows. 
\end{proof}

We can relate these findings back to~\cite[Theorem 12.5]{Jakob} which gives a general connectivity bound for these complexes. For a real number $\nu$, a family of sets $\Delta$ is $AM(\nu)$ if $\Delta$ admits an acyclic matching such that all unmatched sets are of dimension $\lceil \nu \rceil$. For $\lambda = (\lambda_1,...,\lambda_n)$ define $|\lambda| = \sum_{i=1}^{n} \lambda_i$. For a sequence $\mu = (\mu_1,...,\mu_n), n \geq 1,$ define 
\[
\alpha(n, \mu) = \text{min}\{\alpha: BD_n^\lambda \text{ is } AM(\frac{|\lambda| -\alpha}{2} -1)\}.
\]

\begin{theorem} (Thm 12.5, \cite{Jakob}) \label{thm:Jonsson}
Let $G$ be a graph on the vertex set $V$ and $n = |V|$. Let $\{U_1,..,U_t\}$ be a clique partition of $G$ and let $\lambda = (\lambda_1,...,\lambda_n)$ and $\mu = (\mu_1,...,\mu_n)$ be sequences of nonnegative integers such that $\lambda_i \leq \mu_i$ for all $i$. Then $BD_n^\lambda(G)$ is $(\lceil \nu \rceil - 1)$ connected, where 
\[
\nu = \frac{|\lambda|}{2} -\frac{1}{2} \sum\limits_{j=1}^t(\alpha(|U_j|, \mu_{U_j}) -1
\]
\end{theorem}

\begin{proposition}
Theorem~\ref{thm:homotopyType} is an example where Theorem \ref{thm:Jonsson} is not sharp.
\end{proposition}

\begin{proof}
To show this we need to choose a clique partition. By construction of the clawed
graphs, the best we can do is choosing a partition of $2$- and $1$-cliques. Let $\lambda = (2,2,...,2) = \mu$.  By \cite[Lemma 12.6]{Jakob}, all values of $\alpha$ are $2$ and any $\mu$ with $\lambda_i < \mu_i$ for $i = 1,2$ would give rise to larger $\alpha$ values. So, the lower bound on connectivity is given by $ \nu = \frac{|\lambda|}{2} - \frac{1}{2} \sum\limits_{j=1}^{t} 2 -1$. Let $T$ denote the number of claws in $CG$. Since $|\lambda| = 2 |V(CG)|$ and $t = T + (|V| -2T) = \frac{|E|}{3} + (|V| -2\frac{|E|}{3})$, $\nu$ simplifies to $|V| -(\frac{|E|}{3} + |V| -2\frac{|E|}{3}) -1 = \frac{|E|}{3} -1$. From Theorem \ref{thm:homotopyType}, the actual dimension of the $2$-matching complex is $\frac{|E|}{3}$, greater than the lower bound obtained from Theorem \ref{thm:Jonsson}. 
\end{proof}


\section{ Clawed Non-separable Graphs}\label{sec:nonsep}

Suppose we have a graph with potential attaching sites, i.e., vertices of degree $2$. It is natural to ask, given some matching, which of these degree $2$ vertices are actually attaching sites. In addition, once we start attaching leaves, how many can we attach before the $2$-matching complex becomes contractible? To analyze these questions, we will focus our attention on clawed non-separable graphs. Our overall goal of this section will be to maximize the number of attaching sites in a clawed graph by pairing toggle edges in the graph.

\begin{definition}
A \emph{non-separable, i.e., $2$-connected, graph} is a connected graph in which the removal of any $1$ vertex results in a connected graph.  
\end{definition}

Non-separable graphs can be classified through the following construction \cite[Proposition 3.1.1]{Diestel}: 
\begin{enumerate} 
\item Begin with a graph $G:= n$-cycle 
\item Choose $2$ vertices of $G$, say $v_1$ and $v_2$. 
\item Identify the $2$ endpoints of a path of length at least $1$ to $v_1$ and $v_2$ respectively. 
\item Set $G$ to be this new graph and return to (2). 
\end{enumerate}
Stopping after any iteration yields a non-separable graph $G$. Using this construction we can define a clawed non-separable graph.
\begin{definition} 
A \emph{clawed non-separable graph} is a graph obtained through the following construction. 
\begin{enumerate} 
\item Begin with $G$ a clawed $n$-cycle, that is $G := C(C_n)$.  
\item Choose $2$ leaves of $G$, say $v_1$ and $v_2$. 
\item For each endpoint  $x$ in a path $P$ of length at least $1$, let $1$ of the leaves attached to $x$ be an endpoint of the clawed path, $CP$. Identify the $2$ endpoints of a clawed path to $v_1$ and $v_2$ respectively.
\item Set $G$ to be this new graph and return to (2). 
\end{enumerate}
Stopping after any iteration yields a clawed non-separable graph $G$. 
\end{definition} 

Notice that a clawed non-separable graph is the clawed graph of some non-separable graph. We can use the construction of clawed non-separable graphs to get an understanding of the relationship between the number of claws in a clawed non-separable graph and the number of leaves. This will eventually lead us to finding an upper bound for the number of attaching sites in such a graph. Recall that an \emph{attaching site} is a degree $2$ vertex such that attaching a leaf does not affect the homotopy type of the resulting $2$-matching complex. 

\begin{proposition}
Let $T$ be the number of claws in a clawed non-separable graph and $L$ the number of leaves. Then $T$ and $L$  have the same parity modulo $2$. 
\end{proposition}

\begin{proof}
It is clear that for the clawed graph of a non-separable $n$-cycle the parity of $T$ and $L$ is the same. Then, by construction $2$ leaves are chosen, changing the number of leaves but keeping the parity the same. For each additional claw we add another leaf and the parity remains the same. 
\end{proof}

A consequence of this proposition is that there is an even number of possible toggle edges that are not in induced claw units that contain a leaf. Our strategy for obtaining an upper bound for the maximum number of attaching sites will be to \emph{pair} the toggle edges where we say $2$ toggle edges are \emph{paired} if the toggle edges are incident to one another. 

\begin{theorem}\label{thm:upperbound}
Let $CH$ be the clawed graph of a non-separable graph $H$ such that $CH$ has $T$ claws and consider a complete claw-induced partial matching on $CH$.Then, the upper bound for the maximum number of leaves that can be added before changing the homotopy type of $M_2(CH)$ the $2$-matching complex of a clawed non-separable graph is $T$.
\end{theorem}

\begin{proof}
The total number of possible attaching sites is given by $\frac{3T - L}{2}$ because each claw has three vertices with degree less than three, we need to remove the number of leaves, since the degree is $1$, and then divide by $2$ since all remaining vertices are the intersection of $2$ claws. Now to find the maximum number of attaching sites we subtract away the minimum number of vertices that have at least $1$ edge that is toggled on in the complete claw-induced partial matching. 

There is $1$ toggled edge per claw and for any claw that has a leaf we can choose the edge whose endpoint is a leaf as the toggle edge, which will maximize the number of attaching sites since no additional leaf can be added to either endpoint of a leaf edge. The most ideal partial matching pairs the toggled edges, so minimally we have $\frac{T-L}{2}$ vertices that cannot be sites. 

Hence, we have a maximum of $\frac{3T -L}{2} - \frac{T - L}{2} = \frac{2T}{2} = T$ attaching sites. 
\end{proof}

The strategy in the proof of Theorem~\ref{thm:upperbound} was to pair toggle edges as a way to maximize the number of attaching sites. We provide $2$ examples (Figures~\ref{fig:good_example} and ~\ref{fig:bad_example}) in which the toggle edges are depicted with a solid line and the edges in the critical cell are depicted as double lines. In Figure~\ref{fig:good_example}, we have an example of a clawed non-separable graph together with a partial matching which attains the maximum number of attaching sites, namely $5$. 

\begin{figure} [h!]
\begin{tikzpicture}[scale = 0.25, every node/.style={}]
\draw[double, line width=2pt] (0,7) --
(5,7)-- (4,2) -- (1,2) -- (-1,-1) -- (-5, -1) -- (-7,2) -- (-5,5) -- (-7,7) -- (-4,9) -- (0,7);
\draw[line width=1pt] (-4,9) --
(-4,12);
\draw[line width=1pt] (5,7) --
(7,9);
\draw[line width=1pt] (-5,5) --
(-1,5) -- (1,2);
\draw[line width=1pt] (-5,-1) --
(-7,-3);

\fill[] (-7,-3) circle (9pt);
\fill[] (-7,2) circle (9pt);
\fill[] (-7,7) circle (9pt);
\fill[] (-4,12) circle (9pt);
\fill[] (-4,9) circle (9pt);
\fill[] (-5,5) circle (9pt);

\fill[] (-5,-1) circle (9pt);
\fill[] (-1,-1) circle (9pt);
\fill[] (1,2) circle (9pt);
\fill[] (-1,5) circle (9pt);
\fill[] (0,7) circle (9pt);
\fill[] (7,9) circle (9pt);
\fill[] (5,7) circle (9pt);
\fill[] (4,2) circle (9pt);

\end{tikzpicture}
\caption{}
\label{fig:good_example}
\end{figure}


It is not always the case that we can achieve the upper bound for the number of attaching sites for clawed non-separable graphs. In Figure~\ref{fig:bad_example}, we see that after toggling on the leaf edges and doing our best to pair the inner toggle edges we are still left with $2$ independent induced claw units that are surrounded by edges that are already in the critical cell. By Lemma~\ref{lem:clawed} we see no matter which edge we choose in either of these induced claw units as the toggle edge, we will decrease the total number of possible attaching sites and thereby the number of possible attaching sites is less than the maximum.


\begin{figure}[h!]
\begin{tikzpicture}[scale = 0.25, every node/.style={}]

\draw[double, line width=2pt] (14,9.5) --
(11,10)-- (10,7) -- (13,6) -- (13,4) -- (14.5,2) -- (16.5,2);
\draw[double, line width=2pt] (18,9) --
(16.5,8) -- (18,6);
\draw[double, line width=2pt] (16,13) --
(22,6) -- (20,2);
\draw[line width=1pt] (13,6) --
(14.5,8) -- (16.5,8);
\draw[line width=1pt] (14.5,2) --
(13,0);
\draw[line width=1pt] (11,10) --
(10,12);
\draw[line width=1pt] (22,6) --
(24,7.5);
\draw[dashed, line width=1pt] (16,13) --
(16,11) -- (18,9);
\draw[dashed, line width=1pt] (16,11) --
(14,9.5);
\draw[dashed, line width=1pt] (18,6) --
(18,4) -- (20,2);
\draw[dashed, line width=1pt] (18,4) --
(16.5,2);

\fill[] (10,7) circle (9pt);
\fill[] (10,12) circle (9pt);
\fill[] (11,10) circle (9pt);
\fill[] (13,6) circle (9pt);
\fill[] (13,4) circle (9pt);
\fill[] (13,0) circle (9pt);

\fill[] (14.5,2) circle (9pt);
\fill[] (14.5,8) circle (9pt);
\fill[] (14,9.5) circle (9pt);
\fill[] (16,13) circle (9pt);
\fill[] (16,11) circle (9pt);
\fill[] (16.5,8) circle (9pt);
\fill[] (18,6) circle (9pt);
\fill[] (18,4) circle (9pt);
\fill[] (16.5,2) circle (9pt);
\fill[] (20,2) circle (9pt);
\fill[] (22,6) circle (9pt);
\fill[] (24,7.5) circle (9pt);
\fill[] (18,9) circle (9pt);

\end{tikzpicture}
\caption{}
\label{fig:bad_example}
\end{figure}

We end this section with a constructible algorithm to obtain a maximal number of attaching sites in a clawed non-separable graph.

This constructible algorithm to obtain a maximal number of attaching sites prioritizes using leaf edges as toggle edges followed by pairing non-leaf toggle edges. Using Lemma~\ref{lem:claw}, we may arbitrarily choose $1$ of the $3$ edges in each of our claws without changing the homotopy type generated by the claw-induced partial matching. At each step we are bringing together as many of the toggle edges as possible to attain the maximal number of attaching sites. Figure~\ref{fig:algorithm_example} provides an example.
\begin{enumerate}
\item Begin with a clawed $n$-cycle and a claw decomposition $C = \{c_1,...,c_n\}$. Choose all the leaf edges as toggle edges so that all edges in the cycle are in the critical cell.
\item Choose $2$ claws, $c_i$ and $c_j$ to attach the next clawed path. Notice that $c_i$ and $c_j$ are induced claw units that contain a leaf, which we call \emph{leaf-claws}. Modify the $2$-matching on these $2$ leaf-claws so that: 
	\begin{enumerate}
	\item For each of the chosen leaf claws $c_i$ and $c_j$: If the leaf claw is incident to a previously chosen or currently chosen leaf claw change the partial matching to pair the toggle edges of these $2$ leaf-claws, prioritizing the leaf claws incident to only $1$ previously chosen leaf-claw. In doing so the number of attaching sites will either remain the same or increase.
	\end{enumerate}
\item For the new clawed path, let all of the leaf edges be the toggle edges. 
\item Return to (2). 
\end{enumerate}
This algorithm returns the maximum number of attaching sites. Consider taking a claw-induced matching on a clawed non-separable graph. If it was possible to increase the number of attaching sites of by modifying this matching, $1$ of $2$ scenarios may be present: 
\begin{itemize}
\item[(i)] there exists a leaf-claw such that the toggle edge is not the leaf edge, or 
\item[(ii)] there exists a pair of incident claws such that neither $1$ has a toggle edge that is already incident to another toggle edge. 
\end{itemize}
Through this algorithm, all leaves are toggle edges so (i) is not present. Notice that if (ii) appeared in this construction it would arise from step (2) of the algorithm when we add a new clawed path, but during that step we are re-orienting so that whenever possible toggle edges are incident to each other.


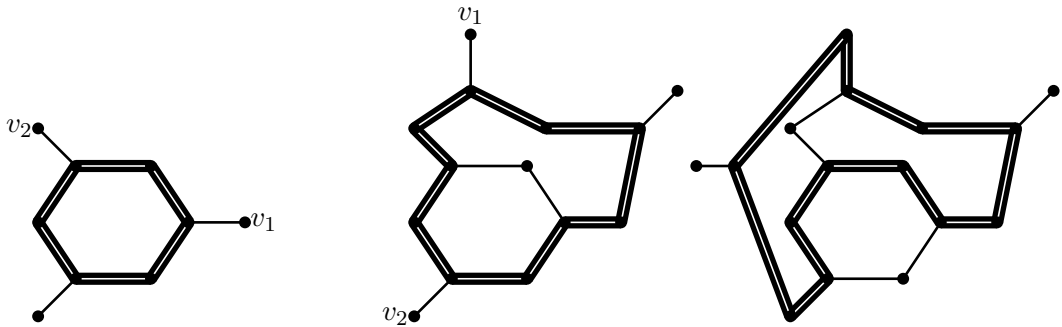
\begin{figure}[hbt!]
\begin{center}
\begin{tikzpicture}[scale = 0.25, every node/.style={}]


\draw[line width = 1pt] (-7,-3) -- (-5,-1); 
\draw[line width = 1pt](-5,5) -- (-7,7); 
\draw[double, line width = 2pt](-5,5) -- (-7,2); 
\draw[double, line width = 2pt](-7,2) -- (-5,-1); 
\draw[double, line width = 2pt](-5,-1) -- (-1, -1); 
\draw[double, line width = 2pt] (-1,-1) -- (1,2); 
\draw[line width = 1pt] (1,2) -- (4,2); 
\draw[double, line width = 2pt] (1,2) -- (-1,5);
\draw[double, line width = 2pt] (-1,5) -- (-5,5);

\fill[] (-7,-3) circle (9pt);
\fill[] (-7,2) circle (9pt);
\fill[] (-7,7) circle (9pt);
\fill[] (-5,5) circle (9pt);

\fill[] (-5,-1) circle (9pt);
\fill[] (-1,-1) circle (9pt);
\fill[] (1,2) circle (9pt);
\fill[] (-1,5) circle (9pt);
\fill[] (4,2) circle (9pt);

\node at (5,2) {$v_1$};
\node at (-8,7) {$v_2$};


\draw[ line width = 1pt] (13,-3) -- (15,-1); 
\draw[ double, line width = 2pt](15,5) -- (13,7); 
\draw[double, line width = 2pt](15,5) -- (13,2); 
\draw[double, line width = 2pt](13,2) -- (15,-1); 
\draw[double, line width = 2pt](15,-1) -- (19, -1); 
\draw[double, line width = 2pt] (19,-1) -- (21,2); 
\draw[double, line width = 2pt] (21,2) -- (24,2); 
\draw[line width = 1pt] (21,2) -- (19,5);
\draw[line width = 1pt] (19,5) -- (15,5); 
\draw[double, line width = 2pt] (13,7) -- (16,9); 
\draw [double, line width = 2pt] (16,9) -- (20,7); 
\draw [double, line width = 2pt] (20,7) -- (25,7); 
\draw [double, line width = 2pt] (25,7) -- (24,2); 
\draw[line width = 1pt] (16,12) -- (16,9); 
\draw[line width = 1pt] (25,7) -- (27,9);

\fill[] (13,-3) circle (9pt);
\fill[] (13,2) circle (9pt);
\fill[] (13,7) circle (9pt);
\fill[] (15,5) circle (9pt);

\fill[] (15,-1) circle (9pt);
\fill[] (19,-1) circle (9pt);
\fill[] (21,2) circle (9pt);
\fill[] (19,5) circle (9pt);
\fill[] (24,2) circle (9pt);

\fill[] (16,12) circle (9pt);
\fill[] (16,9) circle (9pt);
\fill[] (20,7) circle (9pt);
\fill[] (27,9) circle (9pt);
\fill[] (25,7) circle (9pt);

\node at (16,13) {$v_1$};
\node at (12,-3) {$v_2$};


\draw[double, line width = 2pt] (33,-3) -- (35,-1); 
\draw[line width = 1pt](35,5) -- (33,7); 
\draw[double, line width = 2pt](35,5) -- (33,2); 
\draw[double, line width = 2pt](33,2) -- (35,-1); 
\draw[line width = 1pt](35,-1) -- (39, -1); 
\draw[ line width = 1pt] (39,-1) -- (41,2); 
\draw[double, line width = 2pt] (41,2) -- (44,2); 
\draw[double, line width = 2pt] (41,2) -- (39,5);
\draw[double, line width = 2pt] (39,5) -- (35,5); 
\draw[ line width = 1pt] (33,7) -- (36,9); 
\draw [double, line width = 2pt] (36,9) -- (40,7); 
\draw [double, line width = 2pt] (40,7) -- (45,7); 
\draw [double, line width = 2pt] (45,7) -- (44,2); 
\draw[double, line width = 2pt] (36,12) -- (36,9); 
\draw[line width = 1pt] (45,7) -- (47,9); 
\draw[double, line width = 2pt] (33,-3) -- (30,5); 
\draw[double, line width = 2pt] (36,12) -- (30,5);
\draw[line width = 1pt] (30,5) -- (28,5);

\fill[] (33,-3) circle (9pt);
\fill[] (33,2) circle (9pt);
\fill[] (33,7) circle (9pt);
\fill[] (35,5) circle (9pt);

\fill[] (35,-1) circle (9pt);
\fill[] (39,-1) circle (9pt);
\fill[] (41,2) circle (9pt);
\fill[] (39,5) circle (9pt);
\fill[] (44,2) circle (9pt);

\fill[] (36,12) circle (9pt);
\fill[] (36,9) circle (9pt);
\fill[] (40,7) circle (9pt);
\fill[] (47,9) circle (9pt);
\fill[] (45,7) circle (9pt);
\fill[] (30,5) circle (9pt); 
\fill[] (28,5) circle (9pt);

\end{tikzpicture}
\caption{On the left most picture we start with a clawed cycle. Choosing $2$ points, $v_1$ and $v_2$ we attach a clawed path of length $2$. Since the $2$ chosen leaf claws are incident, we pair the toggle edges of each. Then we choose $2$ more vertices, $v_1$ and $v_2$ and continue. In this step there is $1$ claw unit that is incident to $2$ leaf claws and the other claw unit is incident to $1$.}
\label{fig:algorithm_example}
\end{center}
\end{figure}



\section{$k$-matching sequences}\label{sec:wheel}
We now turn our attention to the relationship between $1$-matchings and $2$-matchings. 
Let $G$ be a graph. Define a \emph{$k$-matching sequence} as the sequence $(M_1(G), M_2(G), M_3(G), 
\dots, M_n(G))$, up to homotopy, for $1 \leq k \leq n$ and where $M_n(G)$ is a contractible space. The $n$-matching complex $M_n(G)$ is a cone, hence contractible, precisely when there is an edge $e \in E(G)$ with both endpoints having max degree $n$. 
In this section we will look at the $k$-matching sequence for wheel graphs. 

Let $W_n$ be a wheel graph on $n$ vertices, that is a graph formed by connecting every vertex of a $n-1$ cycle to a single universal vertex. Label the edges of the cycle with $c_0,...,c_{n-2}$ and inner edges by $\ell_0, \ell_1, \dots, \ell_{n-2}$, where $\ell$ is used to symbolize ``leg'' edges, such that $c_i$ shares a vertex with $\ell_{i-1}$ and $\ell_{i}$ modulo $n-1$. See Figure~\ref{fig:wheel_label}.


\begin{figure}[h]
\begin{center}
\begin{tikzpicture}[scale = .75]

\filldraw (-3,3) circle (0.1cm);
\filldraw (-3,0) circle (0.1cm);
\filldraw (1,3) circle (0.1cm);
\filldraw (1,0) circle (0.1cm);
\filldraw (-1,1.5) circle (0.1cm);

\node at (-2,2) {$\ell_0$};
\node at (-0.25, 2.35) {$\ell_1$};
\node at (0.1, 1) {$\ell_2$};
\node at (-1.5,0.75) {$\ell_3$};

\node at (-3.2,1.5) {$c_0$};
\node at (-1, 3.2) {$c_1$};
\node at (1.2, 1.5) {$c_2$};
\node at (-1,-0.2) {$c_3$};

\draw (1,0) -- (-1,1.5) -- (-3,3) -- (-3,0) -- (1,0) -- (1,3) -- (-1,1.5) -- (-3,0) ;
\draw (-3,3) -- (1,3);

\end{tikzpicture}
\end{center}
\caption{$W_5$ and the labeling used in Theorems \ref{thm:wheel_1} and \ref{thm:wheel_2}.}
\label{fig:wheel_label}
\end{figure}
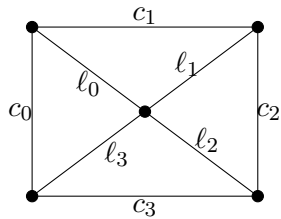


We will determine the homotopy type of the $1$-matching complex and $2$-matching complex of wheel graphs. In the proof of Theorem~\ref{thm:wheel_1}, we will first focus on the ``legs'' or spokes of the wheel and then on the outer cycle. In \cite{Kozlov_trees}, Kozlov proves the following proposition which will come in handy. 

\begin{proposition}[Kozlov,~\cite{Kozlov_trees} Proposition 5.2] \label{thm:Independence_Cycle}
For $n \geq 1$, let $C_n$ denote the cycle of length $n$. The homotopy type of the independence complex of the cycle graph is
\[
Ind(C_n) \simeq 
\begin{cases}
S^{\nu_n} \vee S^{\nu_n} & n \equiv 0 \text{ mod } 3\\
 S^{\nu_n} & n \not\equiv 0 \text{ mod } 3.
\end{cases}
\]
where $\nu_n = \lceil \frac{n-4}{3} \rceil$. 

\end{proposition}

\begin{theorem} \label{thm:wheel_1}
Let $W_n$ be a wheel graph on $n$ vertices. Then, for $k \in \mathbb{N}$, the homotopy type of $M_1(W_n)$ is given by: 
\[
M_1(W_n) \simeq 
\begin{cases}
S^{\nu_n} \vee S^{\nu_n} & n \equiv 1 \text{ mod } 3\\
\bigvee\limits_{n-2} S^{\nu_n} & n \equiv 2 \text{ mod } 3\\
\bigvee\limits_{n} S^{\nu_n}  & n \equiv 0 \text{ mod } 3
\end{cases}
\]
where $\nu_n = \lceil \frac{n-4}{3} \rceil$.
\end{theorem}

\begin{proof}
The strategy of this proof will be to apply the Matching Tree Algorithm on the line graph of $W_n$, see Figure~\ref{fig:linegraph_W_5}. The line graph of $W_n$, denoted $LW_n$ is given by a complete graph on $n-1$ vertices, labeled $\ell_0, \dots, \ell_{n-2}$ and an $(n-1)$-cycle graph $c_0, \dots, c_{n-2}$ with the additional edges $\{c_j, \ell_{j-1}\}$ and $\{c_j, \ell_j \}$ where $j$ is calculated modulo $n-1$. We derive the homotopy type of $M_1(W_n)$ by defining an acyclic (discrete Morse) matching on the face poset of the independence complex of $LW_n$ using the Matching Tree Algorithm.


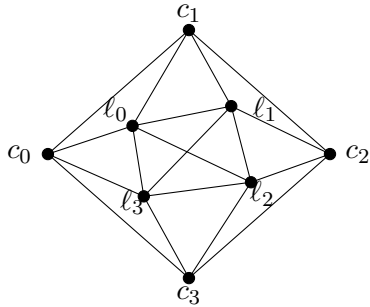
\begin{figure}[h]
\begin{center}
\begin{tikzpicture}[scale = .75]

\draw (-3.5,1.5) -- (-1,3.7) -- (1.5,1.5) -- (-1,-0.7) -- (-3.5,1.5);
\draw (-2,2) -- (-1.8, 0.75) -- (0.1,1) -- (-0.25, 2.35) -- (-1.8,0.75);
\draw(-0.25,2.35) -- (-2,2) -- (0.1,1); 
\draw(-2,2) -- (-1,3.7) -- (-0.25, 2.35) -- (1.5,1.5) -- (0.1,1) -- (-1,-0.7) -- (-1.8, 0.75) -- (-3.5, 1.5) -- (-2,2); 

\filldraw (-3.5,1.5) circle (0.1cm);
\filldraw (-1, 3.7) circle (0.1cm);
\filldraw (1.5, 1.5) circle (0.1cm);
\filldraw (-1,-0.7) circle (0.1cm);

\filldraw (-2,2) circle (0.1cm);
\filldraw (-0.25, 2.35) circle (0.1cm);
\filldraw (0.1, 1) circle (0.1cm);
\filldraw (-1.8,0.75) circle (0.1cm);

\node at (-2.3,2.3) {$\ell_0$};
\node at (0.35, 2.3) {$\ell_1$};
\node at (0.3, 0.8) {$\ell_2$};
\node at (-2,0.65) {$\ell_3$};

\node at (-4,1.5) {$c_0$};
\node at (-1, 4) {$c_1$};
\node at (2, 1.5) {$c_2$};
\node at (-1,-1) {$c_3$};

\end{tikzpicture}
\end{center}
\caption{$LW_5$, the line graph of $W_5$.}
\label{fig:linegraph_W_5}
\end{figure}


Let $P$ denote the face poset of $Ind(L(W_n))$. 
To begin we start with a tentative pivot $\ell_0$ which gives rise to $2$ children $\Sigma(\emptyset; \ell_0)$ and $\Sigma(\ell_0; \ell_1, \dots, \ell_{n-2}, c_0, c_1)$. We first address the right child  $\Sigma(\ell_0; \ell_1, \dots, \ell_{n-2}, c_0, c_1)$. The elements of $V \smallsetminus (A \cup B)$ are $c_2, c_3, \dots c_{n-2}$. Since $c_2$ has exactly $1$ neighbor in $V\smallsetminus (A \cup B)$, use $c_2$ as a pivot leading to $1$ child $\Sigma(\ell_0,c_3; \ell_1, \dots, \ell_{n-2}, c_0,c_1, c_2, c_4)$ where $c_3$ is the matching vertex. Continue in this fashion consecutively choosing the pivot $c_{f(2)}, c_{f(3)}, \dots, c_{f(k)}$ where $3k < n-1$ and $f(i) = j + 3i$ mod $n-1$, with $j$ the index on the tentative pivot of this branch, namely the index of $\ell_j$. At the moment, $j = 2$.  

Notice $\frac{n-1}{3}$ is the number of groups of $3$ that we can break the $(n-1)$ cycle into, where each group consists of $1$ pivot and $2$ neighbors of that pivot. Hence, for $n \equiv 0$ mod $3$ and $n \equiv 2$ mod $3$ $\frac{n-1}{3}$ is not a whole number meaning that all vertices in the outer cycle are either in $A$ or $B$ at the time we reach $c_{f(k)}$. Therefore $\ell_0, c_{f(1)}, c_{f(2)}, \dots, c_{f(k)}$ is the single critical cell of this branch. 

When $n \equiv 1$ mod $3$, $\frac{n-1}{3}$ is a whole number and we have a group of $3$ left over when we reach $c_{f(k)}$, $2$ of which are already in $A$. Therefore, we have an isolated vertex and an empty leaf results, i.e., there are no critical cells of this branch. 

Now, turning our attention to $\Sigma(\emptyset, \ell_0)$ we iterate this process using $\ell_1$ as our tentative vertex. Due to the symmetry of $LW_n$, each branch beginning with $\Sigma(\ell_j, N(\ell_j))$ will either result in an empty leaf or a single critical cell as described above. The general structure of our matching tree can be seen in Figure~\ref{fig:general_structure}.


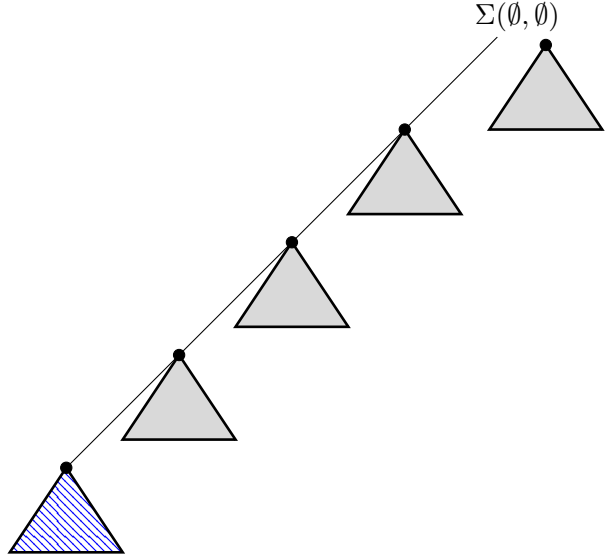
\begin{figure}[h]
\begin{center}
\begin{tikzpicture}[scale = .75]

\draw (3,3) -- (-5,-5);
\fill[pattern=north west lines, pattern color=blue] (-5,-5) -- (-6,-6.5) -- (-4,-6.5) -- (-5,-5);
\draw[line width = 1pt] (-5,-5) -- (-6,-6.5) -- (-4,-6.5) -- (-5,-5);

\fill[thick, gray, opacity=.3] (-3,-3) -- (-2,-4.5) -- (-4,-4.5) -- (-3,-3);
\draw[line width = 1pt] (-3,-3) -- (-2,-4.5) -- (-4,-4.5) -- (-3,-3);

\fill[thick, gray, opacity=.3] (-1,-1) -- (0,-2.5) -- (-2,-2.5) -- (-1,-1);
\draw[line width = 1pt] (-1,-1) -- (0,-2.5) -- (-2,-2.5) -- (-1,-1);

\fill[thick, gray, opacity=.3] (1,1) -- (0,-0.5) -- (2,-0.5) -- (1,1);
\draw[line width = 1pt] (1,1) -- (0,-0.5) -- (2,-0.5) -- (1,1);

\fill[thick, gray, opacity=.3] (3.5,2.5) -- (2.5,1) -- (4.5,1) -- (3.5,2.5);
\draw[line width = 1pt] (3.5,2.5) -- (2.5,1) -- (4.5,1) -- (3.5,2.5);
\filldraw (3.5,2.5) circle (0.1cm);

\filldraw (-5,-5) circle (0.1cm);
\filldraw (-3,-3) circle (0.1cm);
\filldraw (-1,-1) circle (0.1cm);
\filldraw (1,1) circle (0.1cm);
\filldraw[white] (3,3) circle (0.5cm);
\node at (3,3) {$\Sigma(\emptyset, \emptyset)$};

\end{tikzpicture}
\end{center}
\caption{The shaded branches have identical structure with the first element of each branch starting with $A = \emptyset$. The last stripped branch is representative of the outer cycle.}
\label{fig:general_structure}
\end{figure}


Once all vertices of the complete graph have been chosen as tentative vertices, we are left with $1$ child $\Sigma(\emptyset; \ell_0, \dots \ell_{n-2})$ and $V \smallsetminus (A\cup B)$ consists of only vertices on the outside cycle. When $n \equiv 0$ mod $3$, and $m=n-1 \equiv 2$ mod $3$, Proposition~\ref{thm:Independence_Cycle} states there exists a matching tree with $1$ critical cell of size $\nu_n +1$. Additionally, from each of the other branches we have critical cells of size $\nu_n +1$. By Theorem~\ref{thm:main_mta}, the homotopy type is a wedge of spheres, $M_1(W_n) \simeq \bigvee\limits_{n}S^{\nu_n}$ when $n \equiv 0$ mod $3$. 

When $n \equiv 1$ mod $3$, and $m = n-1 \equiv 0$ mod $3$, each of the branches resulting from vertices of the complete graph are empty.
Hence, $M_1(W_n) \simeq M_1(C_m) = Ind(C_m) \simeq S^{\nu_n}$. 

Finally, when $n \equiv 2$ mod $3$, and $m=n-1 \equiv 1$ mod $3$, a subtle shift occurs. Notice that $\nu_n = \nu_m +1$ when $m = n-1$ so Proposition~\ref{thm:Independence_Cycle} says we have $1$ critical cell of size $\nu_n$ and each of the $n-1$ branches gives rise to a critical cell of size $\nu_n +1$. We now argue that we can further match the cells $\alpha:= \{\ell_{n-2}, c_{f(1)}, \dots, c_{f(k)}\}$ and $\beta:=\{c_{f(1)}, \dots, c_{f(k)}\}$. We do so by showing that there exists a linear extension with $u(\beta) = \alpha$, which by Theorem~\ref{thm:linear_extension} gives us that there is an acyclic matching with $\alpha$ and $\beta$ paired, as desired.  

First note that $\{\ell_{n-2}, c_{f(1)}, \dots, c_{f(k)}\}$ is a facet in the independence complex of $LW_n$ for $n \equiv 2$ mod $3$ which means it is a maximal element of the face poset. Since $\beta \prec \alpha \in P$, $\beta$ is a coatom. 

We claim for any pair $(x, u(x))$ for which $\beta <_P x$ or $\beta <_Pu(x)$, i.e., $\beta <_{\mathcal{L}} (x, u(x))$, $\alpha$ is incomparable to $x$ and to $u(x)$. 
If $\beta \prec_P x \prec_P u(x)$, then $\alpha$ is incomparable to $x$ and incomparable to $u(x)$ since $\beta \prec \alpha$ and $\alpha$ is maximal. Suppose $\beta$ is incomparable to $x$ and $\beta <_P u(x)$. Since $\beta \prec \alpha$, $\alpha$ is incomparable to $u(x)$. Since $\beta$ is incomparable to $x$, $\beta < u(x)$, and $x \prec u(x)$ it must be that  $\beta \cup x \subseteq u(x)$. In addition, $\alpha$ and $\beta$ differ by $1$ element and if $x < \alpha$ this would mean $\alpha = \beta \cup x$ which is a contradiction to the incomparability of $u(x)$. 

This means that any pair $(x, u(x))$ in $\mathcal{L}$ such that $\beta < (x,u(x))$ can be moved above $\alpha$. The only concern is if there exists a pair of elements $(y,u(y))$ such that $(y,u(y)) <_{\mathcal{L}} \alpha$ and $(y,u(y)) >_{\mathcal{L}} (x,u(x))$ but this is not possible as this means $(y,u(y)) >_{\mathcal{L}} (x,u(x)) >_{\mathcal{L}} \beta$ and we have seen $(y,u(y))$ is incomparable to $\alpha$. 

Finally, we note that for any pair $(y,u(y))$ such that $ (y,u(y)) <_{\mathcal{L}} \alpha$, we have seen $\beta \not<_{\mathcal{L}} (y,u(y))$ and therefore it is either the case that $(y,u(y))$ is incomparable to $\beta$ or $(y,u(y)) <_{\mathcal{L}} \beta$. 

Hence, we can rearrange $\mathcal{L}$ so that $u(\beta) = \alpha$ which implies pairing $\alpha$ and $\beta$ forms an acyclic matching. It follows from Theorem \ref{thm:main_mta} that $M_1(W_n) \simeq \bigvee\limits_{n-2}S^{\nu_n}$. 
\end{proof}

The next theorem show that for $n \geq 6$, $M_2(W_n)$ is contractible. We need the following lemma \cite[ Lemma 4.3]{Jakob}: 
\begin{lemma} \label{lem:union_matching}
Let $\Delta_0$ and $\Delta_1$ be disjoint families of subsets of a finite set such that $\tau \nsubseteq \sigma$ if $\sigma \in \Delta_0$ and $\tau \in \Delta_1$. If $\mathcal{M}_i$ is an acyclic matching on $\Delta_i$ for $i=0,1$ then $\mathcal{M}_0 \cup \mathcal{M}_1$ is an acyclic matching on $\Delta_0 \cup \Delta_1$. 
\end{lemma}

\begin{theorem} \label{thm:wheel_2}
Let $W_n$ be a wheel graph on $n$ vertices. Then, for $k \in \mathbb{N}$, the homotopy type of $M_2(W_n)$ is given by:
\[
M_2(W_n) \simeq 
\begin{cases}
S^2 \vee S^2 \vee S^2 & n = 4\\
S^3 \vee S^3 & n=5\\
\text{pt} & n \geq 6.
\end{cases}
\]
\end{theorem}  

\begin{proof}
Let $P_n$ be the face poset of $M_2(W_n)$. See figure~\ref{fig:wheel_label} for an example of the labeling of $W_n$.
Our strategy will be to define acyclic matchings on subposets of $P_n$ and then apply Theorem~\ref{thm:patchwork}. 
Define $Q_n$ to be a poset on the elements $\{\mathbf{c_0, c_2, \mathcal{R}}\}$ given by the relations $\mathbf{c_0} \prec \mathbf{c_2} \prec \mathbf{\mathcal{R}}$. The target elements in $Q_n$ are in bold to differentiate them from vertices of $W_n$. 
Now, we define the poset map $\Gamma_n: P_n \rightarrow Q_n$ by defining the preimage $\Gamma_n^{-1}(\alpha)$ for each $\alpha \in Q_n$.

\begin{itemize}

\item For $n=4$ let $\Gamma_n^{-1}(\mathbf{\mathcal{R}}) := \{ \{c_1,c_2,\ell_2\}, \{c_2,\ell_2\}, \{c_2,\ell_2, \ell_0\}, \{c_1,\ell_0, \ell_1\}, \{c_2,\ell_1, \ell_2\}\}$.

\item For $n\geq 5$ let $\Gamma_n^{-1}(\mathbf{\mathcal{R}}) := \{m \in M_2(W_n) | \{c_1,\ell_0, \ell_1\} \subseteq m \text{ or } \{c_1,\ell_0, c_3, \ell_2\} \subseteq m \text{ or }$\\
 $\{c_{n-2},\ell_{n-2}, c_1,\ell_1\}$ $\subseteq m$ $\text{ or }$ $\{c_{n-2},\ell_{n-2},\ell_2\} \subseteq m \}. $
\item  $\Gamma_n^{-1}(\mathbf{c_2}) := \{m \in M_2(W_n) | \{c_1, \ell_0\} \subseteq m \text{ or } \{c_{n-2}, \ell_{n-2}\} \subseteq m\} \smallsetminus \Gamma_n^{-1}(\mathbf{\mathcal{R}})$
\item $\Gamma_n^{-1}(\mathbf{c_0}) = \{m \in M_2(W_n) | \{c_0\} \subseteq m \text{ or } m \cup \{c_0\} \in M_2(W_n) \}. $
\end{itemize}
Since every maximal $2$-matching of $W_n$ either contains $c_0$, $\{c_1,\ell_0\}$, or $\{c_{n-2}, \ell_{n-2}\}$, elements of $P_n$ have been assigned an image under $\Gamma_n$ and, by definition, $\Gamma_n$ is order-preserving poset map. 
For the preimages $\Gamma_n^{-1}(\mathbf{c_0})$ and $\Gamma_n^{-1}(\mathbf{c_2})$ perform a toggle on $c_0$ and $c_2$, respectively. That is, for each $\sigma \in \Gamma_n^{-1}(\alpha)$ that does not contain $\alpha$, pair $\sigma$ with $ \sigma \cup \{\alpha\}$. By Lemma~\ref{lem:toggle}, these matchings are acyclic. In addition, both of these toggles result in a perfect (discrete Morse) matching.
Notice that what remains are the elements of $\Gamma_n^{-1}(\mathbf{\mathcal{R}})$ which is a set of disjoint subposets for $n \geq 5$ where each of the sets $\{c_1,\ell_0, \ell_1\}, \{c_1,\ell_0, c_3, \ell_2\}, \{c_{n-2},\ell_{n-2}, c_1,\ell_1\}$, and $\{c_{n-2},\ell_{n-2}, c_3,\ell_2\}$ are the unique minimal elements of the respective subposets. Since the (poset) join between any $2$ of these elements would contain more than $2$ leg edges, which is not possible in a $2$-matching, these posets are pairwise disjoint. 

\underline{Claim}: Each subposet either consists of $1$ element or is associated to a contractible subcomplex for $n \geq 4$. 

Recall that any subset of edges in a disjoint union of paths forms a $2$-matching. Each of the sets  $\{c_1,\ell_0, \ell_1\}, \{c_1,\ell_0, c_3, \ell_2\}, \{c_{n-2},\ell_{n-2}, c_1,\ell_1\}$, and $\{c_{n-2},\ell_{n-2}, c_3\ell_2\}$ contains at most $2$ leg edges and $2$ cycle edges. Hence the possible edges that we union with any of these elements to form a $2$-matching form a disjoint union of paths when $n \geq 6$. When $n \geq 7$, toggling on $\mathbf{c_4}$ will pair away all of the remaining cells since $\mathbf{c_4}$ can be in any $2$-matching containing the sets $\{c_1,\ell_0, \ell_1\}, \{c_1,\ell_0, c_3, \ell_2\}, \{c_{n-2},\ell_{n-2}, c_1,\ell_1\}$, and $\{c_{n-2},\ell_{n-2}, c_3\ell_2\}$. For $n = 6$, toggles can be made with $\mathbf{c_1, c_3}$ and $\mathbf{c_4}$. Therefore, by Lemma~\ref{lem:union_matching}, $M_2(W_n) \simeq \text{ pt}$ when $n \geq 6$. 

When $n=5$, $\Gamma_5^{-1}(\mathcal{R}) = \{\{c_1,\ell_1,\ell_0\},  \{c_1, \ell_0, \ell_1, c_4\}, \{ c_1,\ell_0, \ell_1, c_3\},$ 
$\{ c_1, \ell_0, \ell_1, c_3, c_4\},$ $\{c_1,\ell_0,\ell_2,c_3\},$ $\{ \ell_0, c_1, \ell_2, c_3, c_4\}, \{c_4, \ell_4, c_1, \ell_1\}, \{ c_4, \ell_4, c_3, \ell_2\}, \{c_4, \ell_4, c_3, \ell_2, c_1\} \}$. Toggling on $\mathbf{c_1}$ and $\mathbf{c_3}$ leaves $2$ critical $3$-cells, namely $\{c_1,\ell_0,\ell_2,c_3\}, \{c_1,\ell_1,\ell_3,c_3\}$. Hence, $M_2(W_5) \simeq S^3 \vee S^3$. 
When $n=4$, $\Gamma_4^{-1}(\mathbf{\mathcal{R}}) := \{ \{c_1,c_2,\ell_2\}, \{c_2,\ell_2\}, \{c_2,\ell_2, \ell_0\}, \{c_1,\ell_0, \ell_1\}, \{c_2,\ell_1, \ell_2\}\}$ and toggling on $\mathbf{c_1}$ leaves $3$ critical $2$-cells $\{c_1,\ell_0,\ell_1\}, \{c_2,\ell_1,\ell_2\},$ and $\{c_2,\ell_2,\ell_0\}$. Hence, $M_2(W_4) = S^2 \vee S^2 \vee S^2$. 

\end{proof}

Since $M_3(W_n)$ is contractible, we have the $k$-matching sequence of $W_4$ is $(S^{\nu_n} \vee S^{\nu_n}; S^2 \vee S^2 \vee S^2; \text {pt })$ and for $W_5$ is $(\underset{n-2}{\vee} S^{\nu_n}; S^3 \vee S^3; \text{ pt })$ where $\nu = \lceil\frac{n-4}{3} \rceil$.


\section{Caterpillar graphs}\label{sec:cat}

A \emph{caterpillar graph} is a tree in which every vertex is on a central path or only $1$ edge away from the path. A \emph{perfect $m$-caterpillar of length $n$}, denoted $G_n$ is a caterpillar graph with $m$ legs at each vertex on the central path of $n$ vertices (see Figure~\ref{fig:perfect_cat}). We conclude the paper with a derivation of $2$-matching complexes of perfect $m$-caterpillar graphs.


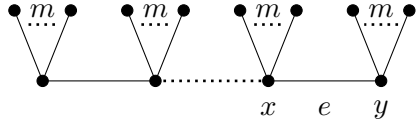
\begin{figure}[h]
\begin{center}
\begin{tikzpicture}[scale = .75]

\draw (-5.5,1.25) -- (-5,0) -- (-4.5,1.25);
\draw (-3.5,1.25) -- (-3,0) -- (-2.5,1.25);
\draw (-1.5,1.25) -- (-1,0) -- (-0.5,1.25);
\draw (0.5,1.25) -- (1,0) -- (1.5,1.25);
\draw (-5,0) -- (-3,0); 
\draw[dotted, line width = 1pt] (-3,0) -- (-1,0); 
\draw (-1,0) -- (1,0);

\draw[dotted, line width = 1pt](-5.25,1) -- (-4.7,1); 
\draw[dotted, line width = 1pt](-3.25,1) -- (-2.7,1); 
\draw[dotted, line width = 1pt](-1.25,1) -- (-0.7,1); 
\draw[dotted, line width = 1pt](1.3,1) -- (0.7,1);

\filldraw (-5,0) circle (0.1cm);
\filldraw (-3,0) circle (0.1cm);
\filldraw (-1,0) circle (0.1cm);
\filldraw (1,0) circle (0.1cm);
\filldraw (-5.5, 1.25) circle (0.1cm); 
\filldraw (-4.5, 1.25) circle (0.1cm); 
\filldraw (-3.5, 1.25) circle (0.1cm); 
\filldraw (-2.5, 1.25) circle (0.1cm);
\filldraw (-1.5, 1.25) circle (0.1cm);
\filldraw (-0.5, 1.25) circle (0.1cm); 
\filldraw (0.5, 1.25) circle (0.1cm); 
\filldraw (1.5, 1.25) circle (0.1cm);

\node at (-5,1.25) {$m$};
\node at (-3,1.25) {$m$};
\node at (-1,1.25) {$m$};
\node at (1,1.25) {$m$};

\node at (-1,-0.5) {$x$};
\node at (1,-0.5) {$y$};
\node at (0,-0.5) {$e$};

\end{tikzpicture}
\end{center}
\caption{A perfect $m$-caterpillar or length $n$.}
\label{fig:perfect_cat}
\end{figure}


In \cite{JelicEtAl}, Jeli\'c Milutinovi\'c et. al. calculate the homotopy type of $M_1(G_n)$ using topological techniques. 

\begin{theorem} \cite[Theorem 5.4]{JelicEtAl}\label{thm:main_M1_cat}
For $m \geq 2$, let $G_n$ be a perfect $m$-caterpillar graph of length $n \geq 1$. Then the homotopy type of $M_1(G_n)$ is given by: 
\begin{equation}
\label{eq:Tn}
    M_1(G_n) \simeq 
    \begin{cases}
        \bigvee\limits_{t=0}^{k} \bigvee\limits_{\alpha_t} S^{k-1+t} & \text{if } n = 2k \\
        \bigvee\limits_{t=0}^k \bigvee\limits_{\beta_t} S^{k+t} & \text{if } n = 2k+1
       \end{cases}
\end{equation} 
where $\alpha_t = \binom{k+t}{k-t}(m-1)^{2t}$ and $\beta_t = \binom{k+1+t}{k-t}(m-1)^{2t+1}$.
\end{theorem}

As we will now see the homotopy type of $M_2(G_n)$ is also a wedge of spheres.

\begin{definition} 
Let $G_n$ be a perfect $m$-caterpillar of length $n$ with the right most edge along the central path $e = \{x_0,x_1\}$, where $y$ is the endpoint of the central path. Define $BD(G_n)$ as the simplicial complex whose vertices are given by edges in $G_n$ and faces are given by subgraphs $H$ of $G_n$ such that the deg$(x_1) \leq 1$ and the degree of any other vertex is at most $2$ in $H$. 
\end{definition} 

 In order to obtain the $2$-matching complex of $G_n$, we will inductively use the bounded degree complex $BD(G_{n-1})$ to build up to $M_2(G_n)$. Namely, our progression will be: 
\[ 
M_2(G_{n-1}) \rightarrow BD(G_n) \rightarrow M_2(G_n) \rightarrow BD(G_{n+1}) \rightarrow M_2(G_{n+1}) \rightarrow \dots. 
\]
Notice that the only difference between $BD(G_{n})$ and $M_2(G_n)$ is the possible degree of the last vertex on the central path. This will allow us to build an inductive argument on $m$-perfect caterpillar graphs.  

For a simplicial complex $\Delta$, let $\Sigma_m(\Delta)$ denote the $m$-point suspension of $\Delta$, that is $\Delta$ join a set of $m$ discrete points.
\begin{lemma} \label{lem:BD_homotopy}
$BD(G_n) \cong \Sigma_m(M_2(G_{n-1})) \vee \Sigma(BD(G_{n-1}))$ 
\end{lemma}

\begin{proof} 
Let $m$ denote the number of legs off of each vertex along the central path as seen in Figure \ref{fig:perfect_cat}. 
For a bounded degree complex $BD(G_n)$, let $e = \{x_0, x_1\}$ be the right most edge along the central path and consider subgraphs $H$ such that deg$(x_1) \leq 1$ in $H$. We can decompose these bounded degree subgraphs into those that contain $e$ and those that do not. Namely, if we exclude $e$, the bounded degree graphs are given by $M_2(G_{n-1}) \ast M_1(St_m)$ where $St_m$ is a star graph on $m$ edges,  and if we include $e$ the bounded degree subgraphs are given by $e \ast BD(G_{n-1})$. These $2$ complexes share $BD(G_{n-1})$ as a common subcomplex and hence 
\[
BD(G_n) \cong M_2(G_{n-1}) \ast M_1(St_m) \underset{BD(G_{n-1})}{\bigcup} e \ast BD(G_{n-1}).
\]
Since $e \ast BD(G_{n-1})$ is a contractible space we get 
\[
BD(G_n) \cong M_2(G_{n-1}) \ast M_1(St_m) \underset{BD(G_{n-1})}{\bigcup} e \ast BD(G_{n-1})/e \ast BD(G_{n-1}) \cong \Sigma_m(M_2(G_{n-1}))/BD(G_{n-1}).
\]
Since $BD(G_{n-1}) \subseteq M_2(G_{n-1})$ we see that  $BD(G_{n-1})$ is contractible in $\Sigma_m(M_2(G_{n-1}))$. Hence, 
\[ BD(G_n) \simeq \Sigma_m(M_2(G_{n-1})) \vee \Sigma(BD(G_{n-1})).\] 
\end{proof}

\begin{lemma}\label{lem:M2_homotopy}
$M_2(G_n) \cong M_2(G_{n-1}) \ast M_2(St_m) \vee \Sigma(\Sigma_m(BD(G_{n-1})))$
\end{lemma}

\begin{proof} 
Let $m$ be the number of legs off each vertex of the central path as seen in Figure \ref{fig:perfect_cat}.
For the $2$-matching complex $M_2(G_n)$, let $e = \{x_0, x_1\}$ be the right most edge along the central path and consider $2$-matchings of $G_n$. Following the argument analogously to Lemma~\ref{lem:BD_homotopy} we can decompose these $2$-matchings into those that contain $e$ and those that do not. Hence, 
\[
M_2(G_n) \cong M_2(G_{n-1}) \ast M_2(St_m) \underset{BD(G_{n-1}) \ast M_1(St_m)}{\bigcup} e \ast BD(G_{n-1}) \ast M_1(St_m). 
\]
Since $e \ast BD(G_{n-1}) \ast M_1(St_m)$ is contractible, we obtain 
\[ 
M_2(G_n) \cong M_2(G_{n-1}) \ast M_2(St_m) / BD(G_{n-1}) \ast M_1(St_m).
\] 
Further, since $BD(G_{n-1}) \subseteq M_2(G_{n-1})$ and $ M_1(St_m) \subseteq M_2(St_m)$ we get that  $BD(G_{n-1}) \ast M_1(St_m) \subseteq M_2(G_{n-1}) \ast M_2(St_m)$ is contractible and
$M_2(G_n) \cong M_2(G_{n-1}) \ast M_2(St_m) \vee \Sigma(\Sigma_m(BD(G_{n-1})))$. 
\end{proof}

\begin{theorem}\label{thm:main_M2_cat}
Let $G_n$ denote a perfect $m$-caterpillar graph of length $n$ with $m \geq 2$. Then, 
\begin{itemize}
\item[(i)]  the homotopy type of $BD(G_n)$ and $M_2(G_n)$ are wedges of spheres of varying dimensions for all $n \geq 1$,
\item[(ii)] the total number of spheres in $BD(G_{i+1})$ and $M_2(G_{i+1})$ is given by the coefficient of $t^i$ in the series
\[
\sum\limits_{i \geq 0} \mathcal{A}_it^i = \sum\limits_{j \geq 0} \mathcal{B}_jt^j = \frac{x}{1-(1+y)t - (x^2 -y)t^2} 
\]
where $x = (m-1)$ and $y = \binom{m-1}{2}$, and 
\item[(iii)] $M_2(G_i) \simeq \bigvee\limits_{j \geq 0} \underset{\beta_{i,j}}{\vee} S^{i+j}$ where $\beta_{i,j}$ the number of spheres of dimension $i +j$ is the coefficient of $r^it^j$ in $B(r,t,x,y) =\sum\limits_{i,j \geq 0} b_{i,j} r^it^j= \frac{x}{1-rt-(x^2-y)r^2t^3 -yrt^2}$ where $x = (m-1)$ and $y = \binom{m-1}{2}$. 
\end{itemize} 
\end{theorem}

\begin{proof}
(i) Since $BD(G_1) = M_1(St_m) \simeq \underset{(m-1)}{\vee} S^0$ and $M_2(G_1) = M_2(St_m) \simeq \underset{\binom{m-1}{2}}{\vee} S^1$, (i) follows from Lemmas~\ref{lem:BD_homotopy}, \ref{lem:M2_homotopy}, and \ref{lem:colim}.

(ii) Let $\mathcal{A}_i$ denote the total number of spheres in the homotopy type of $BD(G_{i+1})$ and $\mathcal{B}_i$ be the total number of spheres in the homotopy type of $M_2(G_{i+1})$. From Lemmas \ref{lem:BD_homotopy}, \ref{lem:M2_homotopy}, and \ref{lem:colim} we know $\mathcal{A}_0 = x := (m-1), \mathcal{B}_0 = y:= \binom{m-1}{2},$ and $\mathcal{A}$, $\mathcal{B}$ follow the recursions: 
\begin{equation}\label{eq:alpha}
\mathcal{A}_i = \mathcal{A}_{i-1} + x\mathcal{B}_{i-1} 
\end{equation}
\begin{equation} \label{eq:beta}
\mathcal{B}_i = x\mathcal{A}_{i-1} + y\mathcal{B}_{i-1}. 
 \end{equation}
 Using equations \ref{eq:alpha} and \ref{eq:beta}, we see that $\mathcal{A}_i = (1+y)\mathcal{A}_{i-1} + (x^2-y)\mathcal{A}_{i-2}$. 
 Let $A(t) = \sum\limits_{i \geq 0} \mathcal{A}_it^i$. Multiplying by $(1-(1+y)t - (x^2 -y)t^2)$ and solving we obtain
 \[
 A(t) = \frac{x}{1-(1+y)t + (x^2 -y)t^2}.
 \]
 The argument for $B(t) = \sum\limits_{j \geq 0} \mathcal{B}_it^i$ is analogous.  
 
 (iii) Let $\alpha_{i,j}$ be the total number of spheres of dimension $j$ in $BD(G_{i+1})$ and $\beta_{i,j}$ the total number of spheres of dimension $j$ in $M_2(G_{i+1})$. Using that $BD(G_1) \simeq \underset{(m-1)}{\vee}S^0$ and $M_2(G_1) \simeq \underset{\binom{m-1}{2}}{\vee} S^1$, and Lemmas \ref{lem:BD_homotopy}, \ref{lem:M2_homotopy}, and \ref{lem:colim} we obtain the following initial conditions 
 \[
 \alpha_{0,0} = x:= (m-1) 
 \]
 \[
 \beta_{0,1} = y := \binom{m-1}{2}
 \]
 \[
 \alpha_{0,j} = 0 \text{ for } j \geq 1
 \]
 \[
 \beta_{0,j} = 0 \text{ for } j \geq 2
 \]
 \[
 \alpha_{i,0} = 0 \text{ for } i \geq 1
 \]
 \[
 \beta_{i, 0} = 0 \text{ for } i \geq 0 
 \]
 Additionally, $\alpha_{i,j}$ and $\beta_{i,j}$ follow the recursions 
 \begin{equation}\label{eq:alpha_2}
 \alpha_{i,j} = \alpha_{i-1, j-1} + x(\beta_{i-1, j-1})  
 \end{equation}
 \begin{equation}\label{eq:beta_2}
 \beta_{i,j} = x\alpha_{i-1,j-2} + y\beta_{i-1,j-2}. 
 \end{equation} 
 Using equations \ref{eq:alpha_2} and \ref{eq:beta_2}, we can see that 
 \[
 \beta_{i,j} = \beta_{i-1, j-1} + (x^2 -y)(\beta_{i-2,j-3} + y(\beta_{i-2,j-3}). 
 \]
 Let $B(r,t,x,y) = \sum\limits_{i,j \geq 0}b_{i,j}r^it^j$ and multiply by $1-rt-(x^2-y)r^2t^3 -yrt^2$. When we solve and use the initial conditions we find that 
 \[
 B(r,t,x,y) = \frac{x}{1-rt-(x^2-y)r^2t^3-yrt^2}
 \] 
 and the result follows from substituting $(m-1)$ for $x$ and $\binom{m-1}{2}$ for $y$.  
\end{proof}

\begin{remark}
From Theorem~\ref{thm:main_M2_cat} (iii), notice that the number of spheres in each dimension is given by a polynomial in $x$ and $y$. If we set $x = y = 1$, we can see that the number of terms in the sum given by the coefficient of $r^it^j$ is a binomial coefficient: 
\[
B(r,t,1,1) = \frac{1}{1-rt(1+t)} = \sum\limits_{k \geq 0} r^kt^k (1+t)^k 
\]
and the coefficient of $[r^it^j] = \binom{i}{j-i}$. 

\end{remark}

\section{Future directions.} 
The original motivation for this project was to study $1$-matching complexes through the lens of $k$-matching complexes for $k \geq 2$. We end with a few open questions. Our exploration of $2$-matching complexes led to observations about the flexibility of the homotopy type and how the homotopy type of clawed non-separable graphs changes (or doesn't change) as new leaves are added. One avenue to explore with this problem involves understanding the interaction between clawed non-separable graphs and additional leaves. 
\begin{question}
Ranging over all clawed non-separable graphs, what is the average maximum number of leaves that can be added without affecting the homotopy type of the resulting $2$-matching complex? 
\end{question}

We have already seen that there are some graphs in which the maximum can be obtained and other graphs in where there is an obstruction to doing so. It would be interesting to know if clawed non-separable graphs tend to have structural properties that obstruct obtaining the maximum and can we expect the maximum number of leaves to be evenly distributed over all such graphs. 

We can also ask about properties of graphs more generally. 
\begin{question} 
Given a graph, how can we determine when leaves can be attached without affecting the resulting homotopy type of the 2-matching complex? 
\end{question} 

In Section~\ref{sec:wheel}, we defined the $k$-matching complex of a graph and explored $2$ examples, wheel graphs and perfect caterpillar graphs. Theorems~\ref{thm:main_M1_cat} and \ref{thm:main_M2_cat} show that the homotopy type of $M_1(G_n)$ and $M_2(G_n)$ are both wedges of spheres with combinatorial structure. A future direction of this work would be to further understand the $k$-matching complex of perfect caterpillar graphs and caterpillar graphs in general.

\begin{conjecture}
The $k$-matching complex of every caterpillar graphs are homotopy equivalent to a wedge of spheres. 
\end{conjecture}



\bibliographystyle{plain}
\bibliography{Vega}
\end{document}